\documentclass[a4paper,11pt]{article}

\usepackage[
    a4paper,
    top=1.05in,
    bottom=1.05in,
    left=1.00in,
    right=1.00in
]{geometry}
\usepackage[T1]{fontenc}

\usepackage{amsmath,amsthm,mathtools}
\usepackage{amssymb,mathrsfs}
\usepackage{newtxtext,newtxmath}
\usepackage{microtype}
\microtypesetup{protrusion=true,expansion=true}
\allowdisplaybreaks[4]

\usepackage{enumitem}
\setlist[enumerate]{leftmargin=2.5em,itemsep=0.25em,topsep=0.35em}
\setlist[itemize]{leftmargin=2.2em,itemsep=0.25em,topsep=0.35em}
\usepackage{graphicx}

\usepackage{titlesec}
\titleformat{\section}{\normalfont\Large\bfseries}{\thesection}{0.75em}{}
\titleformat{\subsection}{\normalfont\large\bfseries}{\thesubsection}{0.75em}{}
\titlespacing*{\section}{0pt}{2.6ex plus 0.8ex minus 0.2ex}{1.2ex plus 0.2ex}
\titlespacing*{\subsection}{0pt}{2.0ex plus 0.6ex minus 0.2ex}{0.9ex plus 0.2ex}

\usepackage{fancyhdr}
\fancypagestyle{plain}{%
  \fancyhf{}%
  \fancyfoot[C]{\thepage}%
}

\usepackage[dvipsnames]{xcolor}
\definecolor{linkblue}{RGB}{20,67,128}
\definecolor{citegreen}{RGB}{25,105,80}
\definecolor{urlpurple}{RGB}{100,55,125}
\usepackage[colorlinks=true,
            linkcolor=linkblue,
            citecolor=citegreen,
            urlcolor=urlpurple,
            plainpages=false]{hyperref}
\usepackage{bookmark}
\bookmarksetup{depth=2}
\numberwithin{equation}{section}
\makeatletter
\renewcommand{\theequation}{\thesection.\ifnum\value{equation}<10 0\fi\arabic{equation}}
\makeatother

\newtheorem{theorem}{Theorem}[section]
\newtheorem{proposition}[theorem]{Proposition}
\newtheorem{lemma}[theorem]{Lemma}
\newtheorem{corollary}[theorem]{Corollary}
\theoremstyle{definition}
\newtheorem{definition}[theorem]{Definition}
\theoremstyle{remark}
\newtheorem{remark}[theorem]{Remark}

\newcommand{\Rm}{\operatorname{Rm}}
\newcommand{\Ric}{\mathrm{Ric}}
\newcommand{\Id}{\operatorname{Id}}

\newcommand{\Reg}{\mathcal R}
\newcommand{\Time}{\mathfrak t}
\newcommand{\dd}{\,\mathrm d}
\newcommand{\dV}{\mathrm dV}
\newcommand{\Vol}{\operatorname{Vol}}
\newcommand{\Area}{\operatorname{Area}}
\newcommand{\dist}{\operatorname{dist}}
\newcommand{\rRm}{r_{\Rm}}
\newcommand{\Wone}{d_{W_1}}
\newcommand{\Act}{\mathcal A}
\newcommand{\calX}{\mathcal X}
\newcommand{\eps}{\epsilon}
\newcommand{\muo}{\boldsymbol\mu}
\newcommand{\Jac}{\operatorname{Jac}}

\title{On the tangent flows of noncollapsed Ricci flow limit spaces}
\author{Yu Li\qquad and\qquad Bing Wang}
\date{}

\begin{document}
\maketitle

\begin{abstract}
We study tangent flows of noncollapsed Ricci flow limit spaces in the
sense of Fang--Li. We prove that the regular part of every negative
time slice is geodesically convex: any minimizing geodesic that meets
the regular part has its entire interior contained in it. We also
identify each negative time slice with the metric completion of its
regular Riemannian part, thereby identifying the negative-time tangent
flow with its associated metric soliton.
\end{abstract}

\tableofcontents

\section{Introduction}
\label{sec:introduction}

Tangent flows model singularities of the Ricci flow. Perelman's
entropy monotonicity underlies their Ricci shrinker structure~\cite{Perelman2002}. Much of the earlier work concerns
blow-ups at the first singular time $T<\infty$ of a smooth Ricci
flow under curvature assumptions. Under the Type-I bound
$|\Rm|\le C(T-t)^{-1}$, Naber~\cite{NaberTypeI} established the
Ricci shrinker structure of dilation limits.
Enders--M\"uller--Topping~\cite{EndersMullerTopping} proved that
blow-ups at singular points subconverge to nonflat Ricci shrinkers and identified the corresponding singular
sets. Mantegazza--M\"uller~\cite{MantegazzaMuller} related the
limiting entropy density at a singular point to the entropy of
the shrinker obtained by blowing up there.

In the K\"ahler setting, Perelman established fundamental estimates
for the normalized K\"ahler--Ricci flow in the anticanonical class
of a Fano manifold, including a uniform scalar curvature bound.
For the corresponding unnormalized flow, this yields the scalar
Type-I estimate $|\mathrm R|\le C(T-t)^{-1}$.
Recently, Jian--Song~\cite[Theorem~1.1]{JianSongTypeI} established
this scalar Type-I estimate for finite-time singularities of the
K\"ahler--Ricci flow on arbitrary compact K\"ahler manifolds.
Building on Perelman's estimates,
Chen--Wang~\cite{ChenWangA,ChenWangB} developed a structure and
weak-compactness theory for polarized K\"ahler--Ricci flows and
established subsequential convergence on Fano manifolds.
Their approach rests on the compactness of a moduli space of
bubble models consisting of noncollapsed Calabi--Yau spaces with
mild singularities. These spaces model parabolic blow-ups at
spatial scales tending to zero in the normalized flow. The uniform
scalar curvature bound forces the limiting flows to be scalar-flat
on their regular parts, and the scalar curvature evolution equation
then implies that these regular parts are Ricci-flat and static.

For general Ricci flows under the scalar Type-I bound
$\mathrm R\le C(T-t)^{-1}$, dilation limits may be singular.
Bamler~\cite{BamlerScalarCompactness} proved that these limits
have a Ricci shrinker structure on their regular parts
and singular sets of codimension at least four.
Hallgren~\cite{HallgrenScalarTypeI} established convergence of
the conjugate heat-kernel entropy to the entropy of the singular
shrinker and used the limiting density to characterize the
singular set of the original flow. Thus the scalar curvature assumption permits
singular shrinking models while retaining an entropy description
of the singularities.

Bamler's $\mathbb F$-convergence theory gives a general framework
for Ricci flow limits without Type-I curvature assumptions
\cite{BamlerEntropy,BamlerCompactness,BamlerStructure}.
Its limiting objects are metric flows, equipped with slice
metrics and conjugate heat kernels. Under entropy noncollapsing,
the structure theory gives a smooth regular Ricci flow spacetime
and quantitative control of the singular set. In particular,
constant pointed Nash entropy yields a shrinking metric soliton:
its regular part satisfies the Ricci shrinker equation,
and its negative-time metric flow is determined by a single
metric measure slice~\cite[Theorem~2.18]{BamlerStructure}.
Potential, volume, and Sobolev estimates for these metric solitons
were developed by Chan--Ma--Zhang~\cite{ChanMaZhangSolitons}.

Fang--Li~\cite{FangLiStructure} introduced a spacetime distance
and studied convergence in the pointed Gromov--Hausdorff sense
for this distance.
Consider smooth closed $n$-dimensional Ricci flows with entropy
bounded below by a common constant $-Y$. A spacetime distance is
defined by comparing conjugate heat kernel measures at earlier
times in Wasserstein distance, together with the parabolic time
separation. For the precise definition, see
\eqref{eq:smooth-spacetime-distance}. A pointed Gromov--Hausdorff
limit of these spacetimes
is a noncollapsed Ricci flow limit space $(Z,d_Z,p_\infty,\Time)$.
At a singular point $z\in Z$ away from the initial time, a tangent
flow $(Z',d_{Z'},z',\Time')$ is a pointed limit
\[
 (Z,r_j^{-1}d_Z,z,r_j^{-2}(\Time-\Time(z)))
 \longrightarrow (Z',d_{Z'},z',\Time'),\qquad r_j\searrow0;
\]
see \cite[Definitions~3.21 and~7.8]{FangLiStructure}.
In particular, $\Time'(z')=0$. The regular part carries a Ricci
flow spacetime structure
\[
 (\Reg',\Time',\partial_{\Time'},g^{Z'}),\qquad
 \Reg'_t=\Reg'\cap Z'_t.
\]
Throughout the paper, we consider such tangent flows and write
$\muo\ge-Y$ for their constant pointed Nash entropy.

Each tangent flow has an associated Bamler shrinking metric
soliton, whose regular negative-time spacetime identifies with
$\Reg'_{(-\infty,0)}$. This soliton describes the geometry before
time zero. The spacetime tangent also retains the time-zero slice
and, when $z$ lies at an interior time of the original limit
space, positive-time slices. These positive-time slices are the
main distinction from the backward tangent flows in Bamler's
theory. We expect the positive-time part to have a self-similar
structure as well; this question remains for future investigation.

Even at negative times, identifying the regular parts does not
immediately identify the ambient spaces. The associated metric
soliton embeds isometrically into $Z'$, but additional
negative-time points outside its image are not excluded a priori.
The resulting question is whether each negative slice of the
tangent flow is the metric completion of its regular Riemannian
part. Fang--Li proved this identification under a scalar
curvature upper bound in \cite[Theorem~7.19]{FangLiStructure}.
Our first result removes this assumption.

We distinguish the spacetime distance $d_{Z'}$ from the slice distance
$d_t^{Z'}$ of \cite[Definition~6.1]{FangLiStructure}. For $t<0$, it is given by
\begin{equation}
 d_t^{Z'}(q_1,q_2)=\lim_{s\nearrow t}
 \Wone^{d_s}(\nu_{q_1;s},\nu_{q_2;s}).
 \label{eq:slice-distance-definition}
\end{equation}
Here $d_s$ is the complete slice metric of the associated metric
soliton, recalled in Section~\ref{subsec:tangent-background}, and
$\Wone^{d_s}$ denotes the Wasserstein distance with cost $d_s$.
The limit exists by Wasserstein monotonicity and is initially
allowed to take the value $+\infty$. On $\Reg'_t$, for $t<0$,
it agrees with the intrinsic Riemannian distance of $g_t^{Z'}$ by
\cite[Corollary~7.10]{FangLiStructure}. The restriction of $d_{Z'}$ to
a time slice need not equal this Riemannian distance.

\begin{theorem}
\label{thm:completion}
For every $t<0$, the metric space $(Z'_t,d_t^{Z'})$ is the metric
completion of $(\Reg'_t,d_{g_t^{Z'}})$. In particular, $d_t^{Z'}$ is
finite-valued, and $(Z'_t,d_t^{Z'})$ is a proper length space.
If $\calX^{z'}$ is the metric soliton associated with $z'$, its canonical
embedding satisfies
\begin{equation}
 \iota_{z'}\bigl(\calX^{z'}_{(-\infty,0)}\bigr)
       =Z'_{(-\infty,0)}.
 \label{eq:surjectivity}
\end{equation}
\end{theorem}

We call the time $-1$ slice $(Z'_{-1},d_{-1}^{Z'})$, equipped
with its regular metric and soliton potential, a \emph{tangent Ricci
shrinker}. The entire negative-time part of the tangent flow is
determined by its tangent Ricci
shrinker through the soliton evolution and the identification in
Theorem~\ref{thm:completion}; see Section~\ref{subsec:soliton}.

The key to Theorem~\ref{thm:completion} is an exact potential
moment identity. Let $f$ be the soliton potential, normalized by
$\tau(\mathrm R+|\nabla f|^2)=f$, where $\tau=-\Time'$.
On the regular spacetime, the function
\[
 u=\tau\left(f-\frac n2\right)
\]
is caloric. We extend the bounded heat representation formula of
\cite[Lemma~5.22]{FangLiStructure} to functions with growth
$\exp(o(f))$ on compact negative-time intervals. It follows that
\begin{equation}
 \int_{\Reg'_s}f\dd\nu_{q;s}
 =\frac{|t|}{|s|}f(q)
   +\left(1-\frac{|t|}{|s|}\right)\frac n2,
 \qquad q\in\Reg'_t,\quad s<t<0.
 \label{eq:intro-moment}
\end{equation}
For a spacetime Cauchy sequence in the associated metric soliton,
Wasserstein monotonicity and concentration bound the second moments
of its conjugate kernels at a fixed earlier time.
Equation~\eqref{eq:intro-moment} then bounds its pointwise potential.
Transport by the modified soliton flow moves the sequence to a
bounded region of a fixed proper slice, while changing the
spacetime distance by a quantity tending to zero. This proves
completeness on compact negative-time intervals and hence surjectivity
of the embedding.
The potential moment identity thus replaces the scalar curvature
bound in \cite[Theorem~7.19]{FangLiStructure}.

As an application, the completion theorem allows a comparison of
the extrinsic spacetime distance from $z'$ with the intrinsic slice
distance. Fix a regular $H_n$-center $p\in\Reg'_{-1}$ of $z'$ and
let $p_{-\tau}$ be its image under the modified soliton flow. Then
\[
 d_{Z'}(z',q)\asymp
 \sqrt\tau+d_{-\tau}^{Z'}(p_{-\tau},q),\qquad q\in Z'_{-\tau},
\]
with comparison constants depending only on $n$.
The precise statement, together with the corresponding potential
and ball comparisons, is Theorem~\ref{thm:comparison}.
Related comparisons between geometric and heat-kernel neighborhoods
on smooth shrinker flows are given by
Li--Wang~\cite[Theorem~1.4(i), Propositions~5.7 and~5.10]{LiWangHeatII}.

We now turn to the regularity of minimizing geodesics. In
Cheeger--Colding theory, Colding--Naber's H\"older continuity of
tangent cones yields regular convexity for noncollapsed Einstein
limits~\cite[Theorem~1.5]{ColdingNaber}.
Chen--Wang~\cite{ChenWangA} developed the corresponding theory
for singular Calabi--Yau spaces. For noncollapsed limits of smooth
Ricci shrinkers, Li--Li--Wang~\cite{LiLiWangStructure} established
the structure theory, including the identification of the ambient
and intrinsic regular-part distances.
Huang--Li--Wang~\cite[Theorem~1.1]{HuangLiWang} then proved regular
convexity, completing the conifold description of these limits.

For the tangent flows considered here, the regular negative-time
geometry is shrinking. We use the term \emph{Ricci shrinker
conifold} for a Riemannian conifold whose regular part carries a
Ricci shrinker structure, following the Riemannian
interpretation of conifolds in the paragraph after Definition~1.2
on page~3 of Chen--Wang~\cite{ChenWangA}. No complex structure
is required, and the dimension is the real dimension. The conifold
definition also requires a uniform volume gap: every tangent cone
at a singular point has unit-ball volume at most
$(1-\delta_n)\omega_n$, where $\delta_n>0$ depends only on $n$
and $\omega_n$ is the Euclidean unit-ball volume. We prove this gap
in Lemma~\ref{lem:entropy-density}.

Theorem~\ref{thm:completion} identifies each negative slice with
the completion of its regular part. The structure results of
Fang--Li~\cite{FangLiStructure} and
Bamler~\cite[Theorem~2.18]{BamlerStructure} supply the Ricci
shrinker geometry, the codimension-four singular-set estimates,
and conical tangent spaces. The remaining property for the
conifold description is regular convexity. Our second main result
establishes it in the ambient slice metric.

\begin{theorem}[Regular convexity]
\label{thm:convexity}
For every $t<0$, any minimizing geodesic in $(Z'_t,d_t^{Z'})$
that meets $\Reg'_t$ has its entire interior in $\Reg'_t$.
In particular, any two points of $\Reg'_t$ are joined by a smooth
minimizing $g_t^{Z'}$-geodesic contained in $\Reg'_t$.
\end{theorem}

Together, these results show that every tangent Ricci shrinker is a
Ricci shrinker conifold in this Riemannian sense.
Theorem~\ref{thm:convexity} thus answers the conifold question raised
by Chen--Wang~\cite[Remark~4.43]{ChenWangB} in the setting of tangent
Ricci shrinkers.

The construction of minimizing geodesics is the principal new step in
Theorem~\ref{thm:convexity}. We first establish a fixed-basepoint
weak-convexity statement. Set $g=g_{-1}^{Z'}$. For every
$x\in\Reg'_{-1}$, almost every $y\in\Reg'_{-1}$ with respect to
$\dV_g$ is joined to $x$ by a unique regular minimizing geodesic
of length $d_g(x,y)=d_{-1}^{Z'}(x,y)$; see
Theorem~\ref{thm:weak} and its quantitative form in
Theorem~\ref{thm:quantitative}. Weak convexity was also discussed
by Bamler~\cite[Remark~2.17]{BamlerStructure} in connection with
mild singularities of static limits.

The weak-convexity argument uses Chen--Wang's control of curvature
along most minimizing $\mathcal L$-geodesics by excluding
high-curvature regions~\cite[Lemma~3.22]{ChenWangB}.
Their last-entry construction
gives a quantitative estimate for the bad endpoints and yields
regular reduced geodesics in the limit in
\cite[Lemma~3.23]{ChenWangB}.
Related heat-kernel estimates and integral estimates along
$\mathcal L$-geodesics enter Bamler's proof of
mildness~\cite[Theorem~1.4]{BamlerScalarCompactness}.

Our construction combines $\mathcal L$-minimization between
two positive backward times with the soliton projection and the
last-entry argument to obtain regular minimizing geodesics.
We keep the backward-time clock anchored at $\Time'=0$ and consider
\[
 \Act_{1,4}(\beta)
   =\int_1^4\sqrt\tau\bigl(\mathrm{R}+|V|^2\bigr)\dd\tau,
 \qquad \beta'=-\partial_{\Time'}+V.
\]
If $\gamma:[0,1]\to\Reg'_{-1}$ is the soliton projection with the
appropriate parameter, then
\begin{equation}
 \Act_{1,4}(\beta)
   =4f(\gamma(1))-2f(\gamma(0))
      +\int_0^1|\dot\gamma|_g^2\dd s.
 \label{eq:intro-action}
\end{equation}
With the endpoints fixed, the soliton projection reduces
$\mathcal L$-minimization to ordinary energy minimization.
Thus a regular minimizing $\mathcal L$-geodesic projects to an
ordinary minimizing geodesic.

The volume estimates for the high-curvature region come from
Fang--Li~\cite[Proposition~8.19 and Theorems~8.18,~8.22]{FangLiStructure}.
Coarea gives a moving inner boundary with small integrated area.
Projecting the last boundary contacts to the endpoint slice,
a pointwise Jacobian comparison and a uniform relative-speed bound
show that the set of bad endpoints has volume tending to zero.
The speed bound follows from an energy estimate for the projected
curves, whose length functionals converge, after subtracting the
endpoint term, to ordinary energy. Passing to the limit on compact
regular domains gives regular minimizing geodesics for almost every endpoint.

The weak-to-strong argument originates in Colding--Naber's
parabolic smoothing of distance functions~\cite{ColdingNaber}.
We use its development on singular spaces by
Chen--Wang~\cite[Proposition~2.52]{ChenWangA} and its weighted
adaptation by Huang--Li--Wang~\cite[Sections~3--4]{HuangLiWang}.
Before this upgrade, weak convexity already recovers the sharp
quadratic bounds for the potential $f$ and the local
$L^2$-Poincar\'e and Sobolev inequalities on the completed slice.
We establish these consequences in Section~\ref{sec:weak-applications},
following the smooth potential argument of
\cite[Lemma~2.1]{HaslhoferMuller}, the analysis on singular spaces
of Chen--Wang~\cite[Section~2]{ChenWangA}, and the local weighted
estimates in \cite[Section~2]{HuangLiWang} and
\cite[Appendix~A]{LiWangCylinders}. These estimates also supply the
analytic inputs for the upgrade. The RCD structure justifies the
heat-energy estimates on the completion. Uniform comparison of small
balls then passes to limit geodesics, where the entropy--volume-density
identity characterizes regular points. Finally, a no-corner argument
and smooth geodesic uniqueness establish the assertion for every
ambient minimizer; see
Subsections~\ref{subsec:entropy-density}--\ref{subsec:regularity-propagation}.

The paper is organized as follows. Section~\ref{sec:preliminaries}
collects the background and cutoff estimates.
Section~\ref{sec:representation} proves the heat equation
representation and the potential moment identity.
Section~\ref{sec:completion} establishes completion and compares
the slice and spacetime distances. Section~\ref{sec:action}
studies spacetime curves, their solitonic projections, and
Jacobian comparison. Section~\ref{sec:weak} proves weak convexity
by the last-entry argument. Finally,
Section~\ref{sec:weak-applications} derives its analytic consequences
and proves regular convexity.

\paragraph{Acknowledgements.}
Yu Li is supported by the National Key Research and Development
Program of China (Grant No.~2025YFA1018200), NSFC-12522105, and
research funds from the University of Science and Technology of
China and the Chinese Academy of Sciences. Bing Wang is supported
by NSFC-12431003 and a research fund from the University of Science
and Technology of China.

\paragraph{AI declaration.}
ChatGPT was used to assist with drafting, mathematical
exposition, revision, reference checking, and preparation of the figure.
The mathematical ideas are the authors' own. Most of the techniques
are drawn from the authors' previous works and those of other
researchers, as cited in the text. The authors take full responsibility
for the mathematical content and correctness.

\section{Preliminaries}
\label{sec:preliminaries}

We recall the definitions and structure results of
\cite{FangLiStructure} used in this paper.

\subsection{Noncollapsed Ricci flow limit spaces and heat kernels}
\label{subsec:kernel}

Fix $T\in(0,\infty]$. For $T<\infty$, set
\[
 \begin{aligned}
 \mathbb I^{++}&=[-T,0],&\qquad
 \mathbb I^+&=[-0.99T,0],\\
 \mathbb I&=[-0.98T,0],&
 \mathbb I^-&=(-0.98T,0].
 \end{aligned}
\]
For $T=\infty$, all four intervals are $(-\infty,0]$.
Let $(M_i^n,g_i(t))$, $t\in\mathbb I^{++}$, be a sequence of
smooth closed Ricci flows, with time functions $\Time_i(x,t)=t$.
Denote their conjugate heat kernel measures by $\nu^i_{x;s}$.

For~$x,y\in M_i\times\mathbb I^+$, put
$t_+=\max\{\Time_i(x),\Time_i(y)\}$ and
$t_-=\min\{\Time_i(x),\Time_i(y)\}$. Following
\cite[Definition~3.2]{FangLiStructure}, define
\begin{equation}
 d_i^*(x,y)=\inf_{\substack{s\in\mathbb I^+\\s\le t_-}}
       \max\left\{\sqrt{t_+-s},
          \Wone^{g_i(s)}(\nu^i_{x;s},\nu^i_{y;s})\right\}.
 \label{eq:smooth-spacetime-distance}
\end{equation}
Thus the comparison times satisfy $-0.99T\le s\le t_-$ when
$T<\infty$, and $s\in(-\infty,t_-]$ when $T=\infty$.
We restrict $d_i^*$ to $M_i\times\mathbb I$.
A pointed Gromov--Hausdorff limit
\[
 (M_i\times\mathbb I,d_i^*,p_i^*,\Time_i)
       \longrightarrow (Z,d_Z,p_\infty,\Time)
\]
is a \emph{Ricci flow limit space}. It is \emph{noncollapsed} if the
approximating flows have entropy bounded below by a common constant
$-Y$; see \cite[Definition~3.21 and Remark~3.22]{FangLiStructure}.
Throughout this paper, we only consider noncollapsed Ricci flow
limit spaces with entropy bounded below by $-Y$.
The limiting spacetime is a proper metric space, and its time
function satisfies
$|\Time(x)-\Time(y)|\le d_Z(x,y)^2$.
We write $Z_t=\Time^{-1}(t)$ and $Z_J=\Time^{-1}(J)$.

The regular set consists of points admitting approximating sequences
with a uniformly positive curvature radius. By
\cite[Definition~5.1 and Theorem~5.2]{FangLiStructure}, it is open
and carries a Ricci flow spacetime structure
\[
 (\Reg,\Time,\partial_{\Time},g^Z).
\]
The regular set is dense away from the initial time, and convergence
on compact regular domains is smooth for the spatial metrics, time
vector fields, and heat kernels. We use spacetime points $x,y,z$
throughout; their time coordinates are $\Time(x),\Time(y),\Time(z)$.

The limit heat kernel is denoted by $K_Z(x;y)$ for
$\Time(y)<\Time(x)$. It is smooth on regular pairs, nonnegative,
and satisfies the forward heat equation in $x$ and the conjugate
heat equation in $y$. For each $x\in Z$ and $s\in\mathbb I^-$ with
$s<\Time(x)$,
\[
 \dd\nu_{x;s}(y)=K_Z(x;y)\dd V_{g_s^Z}(y),\qquad y\in\Reg_s,
 \qquad \nu_{x;\Time(x)}=\delta_x.
\]
These are probability measures with full mass on $\Reg_s$. The
construction at singular basepoints is independent of the associated
metric flow chosen in its definition; see
\cite[Definitions~5.8 and~5.14, Lemmas~5.9 and~5.19]{FangLiStructure}.
They satisfy the reproduction formula
\[
 \nu_{x;s}=\int_{\Reg_u}\nu_{y;s}\dd\nu_{x;u}(y)
       \qquad(s<u<\Time(x)),
\]
and converge weakly to $\delta_x$ as $s\nearrow\Time(x)$.
For a fixed basepoint $x$, the conjugate equation on the regular
spacetime is
\[
 \bigl(-\partial_{\Time,y}-\Delta_y+\mathrm{R}(y)\bigr)K_Z(x;y)=0.
\]
Here the subscript $y$ indicates the variable on which an operator
acts. The reproduction formula is
\cite[Lemma~5.13 and equation~(5.5)]{FangLiStructure}.

The notion of an $H_n$-center and the associated concentration
estimates originate in Bamler~\cite{BamlerEntropy}.
By \cite[Proposition~5.24 and Definition~5.25]{FangLiStructure},
for each $s\in\mathbb I^-$ with $s<\Time(x)$ there is a regular
$H_n$-center $z_s\in\Reg_s$ with
\begin{equation}
 \int_{\Reg_s}d_{g_s^Z}(y,z_s)^2\dd\nu_{x;s}(y)
       \le H_n(\Time(x)-s).
 \label{eq:center}
\end{equation}
For any such center, \cite[Theorem~5.30]{FangLiStructure} gives
\begin{equation}
 K_Z(x;y)\le\frac{C}{(\Time(x)-s)^{n/2}}
       \exp\left(-\frac{d_{g_s^Z}(y,z_s)^2}{5(\Time(x)-s)}\right),
       \qquad y\in\Reg_s,
 \label{eq:gaussian-kernel}
\end{equation}
where $C=C(n,Y)$. Intrinsic distances between different regular
components are assigned the value $+\infty$.
The Wasserstein distance between two conjugate heat flows is
nondecreasing as the common time increases
\cite[Lemma~5.29]{FangLiStructure}. For $t\in\mathbb I^-$, define
\begin{equation}
 d_t^Z(x,y)=\lim_{s\nearrow t}
       \Wone^{g_s^Z}(\nu_{x;s},\nu_{y;s}),\qquad x,y\in Z_t.
 \label{eq:background-slice-distance}
\end{equation}
This is an extended metric, and $d_Z\le d_t^Z$ on $Z_t$; see
\cite[Definition~6.1 and Lemmas~6.3--6.4]{FangLiStructure}.

\subsection{Tangent flows and associated metric solitons}
\label{subsec:tangent-background}

We consider tangent flows based at singular points
$z\in Z_{\mathbb I^-}\setminus\Reg$. Following
\cite[Definition~7.8]{FangLiStructure}, a \emph{tangent flow} at $z$
is a pointed parabolic metric space $(Z',d_{Z'},z',\Time')$ obtained
as a pointed Gromov--Hausdorff limit
\begin{equation}
 (Z,r_j^{-1}d_Z,z,r_j^{-2}(\Time-\Time(z)))
       \longrightarrow (Z',d_{Z'},z',\Time'),\qquad r_j\searrow0.
 \label{eq:tangent-flow-definition}
\end{equation}
In particular, $\Time'(z')=0$, and its time interval is
\[
 I'=\begin{cases}
       (-\infty,0],&\Time(z)=0,\\
       (-\infty,+\infty),&\Time(z)<0.
     \end{cases}
\]
Its regular Ricci flow spacetime is
\[
 (\Reg',\Time',\partial_{\Time'},g^{Z'}),
  \qquad \Reg'_t=\Reg'\cap Z'_t.
\]

The diagonal realization in \cite[equation~(7.3)]{FangLiStructure}
exhibits $Z'$ as a limit of smooth closed flows $(M_i,g_i(t))$ whose
backward time intervals exhaust $(-\infty,0)$. On compact regular
domains there are time-preserving convergence maps $\phi_i$ for
which the pulled-back metrics and time vector fields converge
smoothly, by \cite[Theorem~5.2]{FangLiStructure}. Moreover, the
scalar lower bound \cite[equation~(2.4)]{FangLiStructure} implies
that, on every fixed compact interval $J\Subset(-\infty,0)$,
\begin{equation}
 \mathrm{R}_{g_i(t)}\ge-\epsilon_i(J)\quad(t\in J),\qquad
 \epsilon_i(J)\longrightarrow0.
 \label{eq:scalar-lower}
\end{equation}

The pointed Nash entropy $\mathcal N_{z'}(\tau)$ is constant for
$\tau>0$; write its value as $\muo\ge-Y$. The metric flow
$\calX^{z'}$ associated with $z'$ is a metric soliton, and its regular
negative-time spacetime identifies with $\Reg'_{(-\infty,0)}$;
see \cite[Propositions~7.3 and~7.9]{FangLiStructure}.
Its negative slices are complete, locally compact length spaces
and are the metric completions of their regular parts, as recalled
in \cite[Section~7]{FangLiStructure}; see also
\cite[Theorems~2.18 and~15.69]{BamlerStructure}.

The canonical embedding
\[
 \iota_{z'}:(\calX^{z'},d^*)\longrightarrow(Z',d_{Z'})
\]
is an isometry onto its image and preserves the time functions,
regular spacetime structures, and conjugate heat kernel measures
\cite[Theorem~4.27 and Proposition~5.6]{FangLiStructure}.
We identify the associated soliton with its image when writing
$\Time'$ or $K_{Z'}$ on it. Its image contains
$\Reg'_{(-\infty,0)}$, which is dense in $Z'_{(-\infty,0)}$.

Let $d_t$ denote the complete slice
metric on $\calX^{z'}_t$. On $\Reg'_t$ it agrees with the intrinsic
Riemannian distance and with $d_t^{Z'}$ by
\cite[Corollary~7.10]{FangLiStructure}.
For $q_1,q_2\in\calX^{z'}$, put
$t_+=\max\{\Time'(q_1),\Time'(q_2)\}$ and
$t_-=\min\{\Time'(q_1),\Time'(q_2)\}$.
The same definition as \eqref{eq:smooth-spacetime-distance}, now
with comparison times extending to $-\infty$ and transport measured
by the complete slice metrics $d_s$, gives
\begin{equation}
 d^*(q_1,q_2)=\inf_{s\le t_-}
       \max\left\{\sqrt{t_+-s},
          \Wone^{d_s}(\nu_{q_1;s},\nu_{q_2;s})\right\}.
 \label{eq:dstar}
\end{equation}
In particular,
\begin{equation}
 \begin{split}
 \sqrt{|\Time'(q_1)-\Time'(q_2)|}&\le d^*(q_1,q_2),\\
 d^*(q_1,q_2)&\le d_t(q_1,q_2)
       \quad\bigl(\Time'(q_1)=\Time'(q_2)=t\bigr).
 \end{split}
 \label{eq:elementary-distances}
\end{equation}

\subsection{Soliton coordinates and potential estimates}
\label{subsec:soliton}

Set $\tau=-\Time'$ on the regular negative-time spacetime.
Choose the additive constant of the soliton potential $f$ so that
\begin{equation}
 \Ric+\nabla^2 f=\frac{g^{Z'}}{2\tau},\qquad
 \partial_{\Time'}f=|\nabla f|^2,\qquad
 \tau\bigl(|\nabla f|^2+\mathrm{R}\bigr)=f.
 \label{eq:soliton-identities}
\end{equation}
Then the conjugate kernel based at $z'$ is
\begin{equation}
 K_{Z'}(z';q)=(4\pi\tau(q))^{-n/2}e^{-f(q)-\muo}.
 \label{eq:normalized-soliton-kernel}
\end{equation}
These formulas follow from
\cite[Proposition~7.9 and equation~(7.7)]{FangLiStructure}
by subtracting the constant entropy from the kernel potential.
Here $\mathrm{R}\ge0$.
We have $f>0$ on the regular spacetime. Indeed, the associated
soliton either has $\mathrm R>0$ or is a Ricci-flat metric cone
\cite[Theorem~2.18]{BamlerStructure}. In the cone case,
$f=d_{-1}(o,\cdot)^2/4$ on the reference slice, where $o$ is the
vertex, which is singular in the non-Euclidean setting considered here.
The soliton identities give
\begin{equation}
 |\nabla f|^2+\mathrm{R}=\frac f\tau,\qquad
 \Delta f+\mathrm{R}=\frac n{2\tau},\qquad
 (\partial_{\Time'}-\Delta)f=\frac{f-n/2}{\tau}.
 \label{eq:f-identities}
\end{equation}
On the reference slice $\Reg'_{-1}$ we set $g=g_{-1}^{Z'}$ and
use $f$ for the restriction of the spacetime potential. By
\cite[Proposition~7.9]{FangLiStructure}, each regular negative-time
slice is connected.

The tangent Ricci shrinker is thus the completed slice
\[
 (X,d,g,f)=(Z'_{-1},d_{-1}^{Z'},g_{-1}^{Z'},f|_{Z'_{-1}}),
\]
where $g$ is defined on $\Reg'_{-1}$ and $f$ extends continuously
to $X$ by \eqref{eq:quadratic-potential} and
$|\nabla\sqrt f|\le1/2$. Its normalization is
\begin{equation}
 \mathrm R+|\nabla f|^2=f,\qquad
 \dd\nu=(4\pi)^{-n/2}e^{-f-\muo}\dV_g,\qquad \nu(X)=1.
 \label{eq:tangent-shrinker-normalization}
\end{equation}
The metric completion here is supplied by
Theorem~\ref{thm:completion}.

Let $\boldsymbol{\psi}^s$ be the flow generated by
$\tau(\partial_{\Time'}-\nabla f)$ on $\Reg'_{(-\infty,0)}$,
as in \cite[Section~7, Lemmas~7.28--7.29]{FangLiStructure}.
Its time and metric scaling are
$\Time'\circ\boldsymbol{\psi}^s=e^{-s}\Time'$ and
$(\boldsymbol{\psi}^s)^*g^{Z'}=e^{-s}g^{Z'}$.
Define
\[
 \boldsymbol{\varphi}_\tau=\left.\boldsymbol{\psi}^{-\log\tau}\right|_{\Reg'_{-1}}:
       \Reg'_{-1}\longrightarrow\Reg'_{-\tau},\qquad \tau>0.
\]
These spacetime maps satisfy
\begin{equation}
 \boldsymbol{\varphi}_\tau^*g_{-\tau}^{Z'}=\tau g,\qquad
 f\circ\boldsymbol{\varphi}_\tau=f,\qquad
 \frac{\mathrm d}{\mathrm d\tau}\boldsymbol{\varphi}_\tau
        =-\partial_{\Time'}+\nabla f,
 \qquad \boldsymbol{\varphi}_1=\Id.
 \label{eq:soliton-pullback}
\end{equation}
The soliton tube over a compact subset of $\Reg'_{-1}$
and a compact interval in $(0,\infty)$ is compact in the regular
spacetime. The maps $\boldsymbol{\psi}^s$ extend to $Z'_{(-\infty,0]}$ by
\cite[Proposition~7.30]{FangLiStructure}, with
\begin{equation}
 \begin{split}
 d_{Z'}(\boldsymbol{\psi}^s(x),\boldsymbol{\psi}^s(y))&=e^{-s/2}d_{Z'}(x,y),\\
 d_{e^{-s}t}^{Z'}(\boldsymbol{\psi}^s(x),\boldsymbol{\psi}^s(y))&=e^{-s/2}d_t^{Z'}(x,y)
       \qquad(x,y\in Z'_t).
 \end{split}
 \label{eq:soliton-distance-scaling}
\end{equation}

Under these identifications, the negative slices carry the metrics
$\tau g$ and the same probability measure $\nu$. The soliton
potential determines the time vector field in these coordinates;
the associated metric soliton supplies the conjugate heat kernels
\cite[Theorems~2.18 and~15.69]{BamlerStructure}.
Together with \eqref{eq:dstar} and Theorem~\ref{thm:completion},
this reconstructs the negative-time tangent flow from its tangent
Ricci shrinker.

Choose a regular $H_n$-center $p\in\Reg'_{-1}$ of $z'$ and put
$p_{-\tau}=\boldsymbol{\varphi}_\tau(p)$. By
\cite[Lemma~7.12]{FangLiStructure}, refining the potential estimates
of \cite[Theorem~1.1]{ChanMaZhangSolitons},
\begin{equation}
 C^{-1}\frac{d_{g^{Z'}_{-\tau}}(p_{-\tau},q)^2}{\tau}-C
 \le f(q)
 \le C\left(1+\frac{d_{g^{Z'}_{-\tau}}(p_{-\tau},q)^2}{\tau}\right),
 \label{eq:quadratic-potential}
\end{equation}
where $C=C(n)$. On a compact interval $J\Subset(-\infty,0)$,
\begin{equation}
 \begin{split}
 C_J^{-1}d_{g_t^{Z'}}(p_t,q)^2-C_J
       &\le f(q)\le C_J\bigl(1+d_{g_t^{Z'}}(p_t,q)^2\bigr),\\
 |\nabla f|&\le C_J\sqrt f,\qquad
 0\le\partial_{\Time'}f\le C_J f,
 \end{split}
 \label{eq:slab-potential}
\end{equation}
with $C_J=C(n,J)$. In particular, on $\Reg'_{-1}$,
\begin{equation}
 |\nabla f|(q)\le\sqrt{f(q)}
       \le\sqrt{f(p)}+\frac{1}{2}d_g(p,q).
 \label{eq:linear-drift}
\end{equation}

\subsection{Curvature radius and cutoff estimates}
\label{subsec:cutoffs}

We use the curvature radius of
\cite[Definition~7.4]{FangLiStructure}. For $z\in\Reg'_t$,
$r_{\Rm}(z)$ is the supremum of $r>0$ such that
$B_{g_t^{Z'}}(z,r)$ is relatively compact in $\Reg'_t$ and the
worldline product domain
\[
 B_{g_t^{Z'}}(z,r)\times\bigl([t-r^2,t+r^2]\cap I'\bigr)
\]
is defined in $\Reg'$ with $|\Rm|\le r^{-2}$. Here $I'$ is the
interval specified in Section~\ref{subsec:tangent-background}.
The product uses the worldlines
of $\partial_{\Time'}$. Set $r_{\Rm}=0$ on $Z'\setminus\Reg'$.
The convergence and spacetime Lipschitz properties of this scale
are \cite[Lemma~7.5 and Proposition~7.6]{FangLiStructure}.

The following estimates are collected from
\cite[Proposition~8.19, Theorems~8.18 and~8.22]{FangLiStructure}.

\begin{lemma}[Singular cutoffs and integral bounds]
\label{lem:cutoff-package}
There are smooth functions $\eta_\rho$ on the regular negative-time
spacetime, taking values in $[0,1]$, such that
\begin{equation}
 \begin{gathered}
 \eta_\rho>0\ \Longrightarrow\ \rRm\ge\rho,\qquad
 \eta_\rho=1\text{ on }\{\rRm\ge2\rho\},\\
 \rho|\nabla\eta_\rho|+
 \rho^2\bigl(|\partial_{\Time'}\eta_\rho|+|\nabla^2\eta_\rho|\bigr)
       \le C.
 \end{gathered}
 \label{eq:singular-cutoffs}
\end{equation}
Their supports on bounded regions of compact negative-time slabs
are relatively compact in the regular spacetime. For every
$\theta>0$, every compact interval $J\Subset(-\infty,0)$, and
every bounded spacetime region $Q\subset Z'_J$, we have, for all
sufficiently small $\rho>0$,
\begin{equation}
 \begin{split}
 \int_J\int_{Q_t\cap\{\rRm<2\rho\}}\dd V_{g_t^{Z'}}\dd t
       &\le C_{Q,J,\theta}\rho^{4-\theta},\\
 \sup_{t\in J}\Vol_{g_t^{Z'}}\bigl(Q_t\cap\{\rRm<2\rho\}\bigr)
       &\le C_{Q,J,\theta}\rho^{4-\theta}.
 \end{split}
 \label{eq:parabolic-bad-volume}
\end{equation}
Here $Q_t=Q\cap\Reg'_t$.
On the reference slice, for $A\ge1$,
\begin{equation}
 \Vol_g\bigl(B_g(p,A)\cap\{\rRm<2\rho\}\bigr)
       \le C_{A,\theta}\rho^{4-\theta}.
 \label{eq:slice-bad-volume}
\end{equation}
Finally, for every $J\Subset(-\infty,0)$ there is $C_J<\infty$
such that, for $A\ge2$,
\begin{equation}
 \sup_{t\in J}\Vol_{g_t^{Z'}}\{f\le A^2\}\le C_J A^n.
 \label{eq:sublevel-volume}
\end{equation}
\end{lemma}

\begin{proof}
The cutoff properties and the first bound in
\eqref{eq:parabolic-bad-volume} follow from
\cite[Proposition~8.19 and equation~(8.19)]{FangLiStructure}.
For the second bound, sweep
$Q_t\cap\{\rRm<2\rho\}$ through a fixed small interval of soliton
times using $\boldsymbol{\psi}^s$. The swept set lies in a fixed bounded
negative-time region and has curvature radius at most $C_J\rho$.
Moreover, $\Time'\circ\boldsymbol{\psi}^s=e^{-s}\Time'$ and the slice volume
scales by $e^{-ns/2}$, so its spacetime volume is bounded below by
a fixed multiple of the original slice volume, uniformly for
$t\in J$. Applying the parabolic bound to this swept set gives
the exponent $4-\theta$ on every slice.
The reference-slice estimate \eqref{eq:slice-bad-volume} is also
\cite[Theorem~8.22(ii)]{FangLiStructure}.
Finally, \eqref{eq:sublevel-volume} follows from
\cite[Proposition~5.36]{FangLiStructure} and
\eqref{eq:slab-potential}, using soliton scaling to make the
constant uniform on $J$.
\end{proof}

\section{Heat equation representation}
\label{sec:representation}

We extend the bounded heat equation representation of
\cite[Lemma~5.22]{FangLiStructure} and
\cite[Theorem~15.28(d)]{BamlerStructure} to functions of
Gaussian-subcritical growth. The corresponding smooth-shrinker
analysis, including potential cutoffs, maximum principles, and
heat-kernel constructions, is developed in Li--Wang~\cite{LiWangHeat}.

\begin{definition}
\label{def:subcritical}
A function $u$ on $\Reg'_J$, for $J\Subset(-\infty,0)$, has
\emph{Gaussian-subcritical growth} if, for every $\eps>0$, there is
$C_{J,\eps}<\infty$ such that
\begin{equation}
 |u(q)|\le C_{J,\eps}e^{\eps f(q)}\qquad(q\in\Reg'_J).
 \label{eq:subcritical-growth}
\end{equation}
By \eqref{eq:slab-potential}, this is equivalently growth of order
$\exp(o(d_{g_t^{Z'}}(p_t,\cdot)^2))$, uniformly on $J$.
In particular, every polynomial-growth function satisfies this
condition.
\end{definition}

\begin{theorem}[Gaussian-subcritical heat representation]
\label{thm:representation}
Let $a<b<0$ and let
\[
 u\in C^0(\Reg'_{[a,b]})\cap C^\infty(\Reg'_{(a,b]})
\]
satisfy $(\partial_{\Time'}-\Delta)u=0$ and
\eqref{eq:subcritical-growth} on $[a,b]$. For every
$q\in\Reg'_t$, $a<t\le b$,
\begin{equation}
 u(q)=\int_{\Reg'_a}u(y)\dd\nu_{q;a}(y).
 \label{eq:heat-representation}
\end{equation}
\end{theorem}

\subsection{Kernel decay and first derivative estimates}

Fix $q\in\Reg'_t$ and an interval
$I=[\alpha,\beta]\Subset(a,t)$. In this section we write
$K(y)=K_{Z'}(q;y)$, with $y$ a spacetime point.
For $y\in\Reg'_s$, define the parabolic cylinder
\[
 P(y,r)=\{x_\sigma:x\in B_{g_s^{Z'}}(y,r),\ |\sigma-s|<r^2\},
\]
whenever this domain is defined. Here $x_\sigma$ is the point at
time $\sigma$ on the worldline of $x$ generated by
$\partial_{\Time'}$.

\begin{lemma}
\label{lem:kernel-decay}
For some $c,C>0$,
\begin{equation}
 K(y)\le Ce^{-cf(y)}\qquad(y\in\Reg'_s,\ s\in I).
 \label{eq:kernel-potential-decay}
\end{equation}
Moreover, after decreasing $c$, for every $\eps>0$,
\begin{equation}
 \begin{split}
 |\nabla u|(y)&\le C_\eps(1+\rRm(y)^{-1})e^{\eps f(y)},\\
 |\nabla K|(y)&\le C(1+\rRm(y)^{-1})e^{-cf(y)}.
 \end{split}
 \label{eq:gradient-envelopes}
\end{equation}
\end{lemma}

\begin{proof}
Let $z_s$ be an $H_n$-center of $q$. The center inequality and
Wasserstein monotonicity give
\begin{align*}
 d_s(z_s,p_s)
 &\le \Wone^{d_s}(\delta_{z_s},\nu_{q;s})
       +\Wone^{d_s}(\nu_{q;s},\nu_{z';s})
       +\Wone^{d_s}(\nu_{z';s},\delta_{p_s})\\
 &\le C_I+\Wone^{d_\beta}(\nu_{q;\beta},\nu_{z';\beta}).
\end{align*}
The last term is finite. The reference centers $p_s$ have uniformly
bounded second moments for $\nu_{z';s}$ on $I$, by soliton scaling.
Thus the centers $z_s$ remain at bounded distance from $p_s$.
Equations~\eqref{eq:gaussian-kernel} and
\eqref{eq:slab-potential} imply
\eqref{eq:kernel-potential-decay}.

For $y\in\Reg'_s$, $s\in I$, set
\[
 r=\frac{1}{2}\min\{1,\rRm(y),\dist(I,\{a,t\})^{1/2}\}.
\]
The cylinders $P(y,3r/2)$ lie in the regular spacetime over a fixed
compact interval $I^+\Subset(a,t)$. Apply the preceding kernel
bound on $I^+$. The curvature-radius bound and Shi's estimates on
these larger cylinders control the rescaled coefficients on
$P(y,r)$. By \eqref{eq:quadratic-potential}--\eqref{eq:slab-potential}
and standard local estimates,
\begin{equation}
 C^{-1}f(y)-C\le f(w)\le Cf(y)+C
       \qquad(w\in P(y,r)),
 \label{eq:local-potential-comparison}
\end{equation}
with uniform constants. Interior estimates for the heat equation give
\[
 |\nabla u|(y)\le Cr^{-1}\sup_{P(y,r)}|u|.
\]
Apply \eqref{eq:subcritical-growth} with a smaller exponent and use
\eqref{eq:local-potential-comparison} to obtain the first bound in
\eqref{eq:gradient-envelopes}. The same comparison gives
$\sup_{P(y,r)}K\le Ce^{-c'f(y)}$ for some $c'>0$.
Interior estimates for the conjugate heat equation, with the time
variable reversed, give the second bound.
\end{proof}

\subsection{The two cutoff limits}

We first record the integrability needed below. For every compact
$J\Subset(-\infty,0)$ and $A\ge2$,
\begin{equation}
 \int_J\int_{\Reg'_s\cap\{f\le A^2\}}
       (1+\rRm^{-1})\dd V_{g_s^{Z'}}\dd s\le C_J A^n.
 \label{eq:inverse-curvature-integral}
\end{equation}
Indeed, \cite[Theorem~8.22(iii)]{FangLiStructure}, with its
parameter $\epsilon=1/2$, bounds the integral of $\rRm^{-3}$
over $B_g(p,L)$ by $CL^{n-1}$ for $L\ge1$.
Combine $\rRm^{-1}\le1+\rRm^{-3}$ with
\eqref{eq:sublevel-volume}, \eqref{eq:slab-potential}, and soliton
scaling to obtain \eqref{eq:inverse-curvature-integral}.

Let $\chi:[0,\infty)\to[0,1]$ be smooth, equal to one on $[0,1]$
and zero on $[2,\infty)$. For $A\ge2$, put
\[
 \chi_A=\chi(f/A^2),\qquad
 \Theta_{\rho,A}=\eta_\rho\chi_A.
\]
The derivatives of $\chi_A$ are supported on
$\mathcal E_A=\{A^2\le f\le2A^2\}$ and satisfy
\begin{equation}
 |\nabla\chi_A|\le C A^{-1},\qquad
 |\partial_{\Time'}\chi_A|\le C.
 \label{eq:outer-cutoff-bounds}
\end{equation}
Set
\[
 I_{\rho,A}(s)=\int_{\Reg'_s}\Theta_{\rho,A}uK\dd V_{g_s^{Z'}}.
\]
Using the cutoffs of Lemma~\ref{lem:cutoff-package} and the heat
and conjugate heat equations, integration by parts gives
\begin{equation}
 \frac{\mathrm d}{\mathrm ds}I_{\rho,A}(s)
 =\int_{\Reg'_s}\left(
 uK\partial_{\Time'}\Theta_{\rho,A}
 -K\langle\nabla\Theta_{\rho,A},\nabla u\rangle
 +u\langle\nabla\Theta_{\rho,A},\nabla K\rangle
 \right)\dd V_{g_s^{Z'}}.
 \label{eq:green-identity}
\end{equation}

First fix $A$ and let $\rho\searrow0$. On the transition set of
$\eta_\rho$ inside $\{f\le2A^2\}$,
\[
 |u|+K\le C_A,\qquad
 |\nabla u|+|\nabla K|\le C_A\rho^{-1}.
\]
Hence the spacetime integral of the terms in
\eqref{eq:green-identity} differentiating $\eta_\rho$ has absolute
value at most
\[
 C_{A,\theta}\rho^{-2}\rho^{4-\theta}
       =C_{A,\theta}\rho^{2-\theta}\longrightarrow0
 \qquad(0<\theta<2).
\]
At the two fixed endpoints $\alpha,\beta$, dominated convergence
applies, since the potential sublevel has finite volume. The
remaining spacetime integrands are dominated on that sublevel by
$C_A(1+\rRm^{-1})$, which is integrable by
\eqref{eq:inverse-curvature-integral}. Thus, if
\[
 I_A(s)=\int_{\Reg'_s}\chi_AuK\dd V_{g_s^{Z'}},
\]
then
\begin{equation}
 \begin{split}
 I_A(\beta)-I_A(\alpha)
 =\int_\alpha^\beta\int_{\Reg'_s}\bigl(
 &uK\partial_{\Time'}\chi_A
 -K\langle\nabla\chi_A,\nabla u\rangle\\
 &+u\langle\nabla\chi_A,\nabla K\rangle
 \bigr)\dd V_{g_s^{Z'}}\dd s.
 \end{split}
 \label{eq:green-after-singular}
\end{equation}

Now choose the growth exponent $\eps$ smaller than the kernel decay
constants in Lemma~\ref{lem:kernel-decay}. On $\mathcal E_A$,
\[
 |u|K+K|\nabla u|+|u||\nabla K|
       \le C(1+\rRm^{-1})e^{-cA^2}.
\]
Using \eqref{eq:outer-cutoff-bounds} and
\eqref{eq:inverse-curvature-integral}, the right side of
\eqref{eq:green-after-singular} is bounded in absolute value by
\[
 C A^n e^{-cA^2}\longrightarrow0
       \qquad\text{as }A\nearrow\infty.
\]
The slice-volume bound and Gaussian decay also give convergence of
the endpoint integrals. Letting $A\nearrow\infty$, we obtain
\begin{equation}
 \int_{\Reg'_\beta}u\dd\nu_{q;\beta}
       =\int_{\Reg'_\alpha}u\dd\nu_{q;\alpha},
       \qquad a<\alpha<\beta<t.
 \label{eq:constant-pairing}
\end{equation}

\subsection{The initial and terminal limits}

At the initial endpoint, insert fixed cutoffs
$\eta_\rho\chi_A$ on a short interval $[a,a+\delta]$.
Continuity of $u$, smoothness of $K$, and compactness of the regular
support imply convergence of the truncated integral as $s\searrow a$.
For fixed $A$, the omitted singular integral is bounded by
$C_{A,\theta}\rho^{4-\theta}$, uniformly in this short interval, by
\eqref{eq:parabolic-bad-volume}. The omitted tail is uniformly bounded
by a Gaussian tail in $f$, since these times stay separated from
$t$. Removing the two cutoffs yields
\[
 \lim_{s\searrow a}\int_{\Reg'_s}u\dd\nu_{q;s}
       =\int_{\Reg'_a}u\dd\nu_{q;a}.
\]

At the terminal endpoint, let $q_s$ be the ordinary worldline of
$q$ in a regular cylinder about $q$, and put $h=t-s$.
For a regular $H_n$-center $z_s$ of $q$,
\cite[Lemma~5.31]{FangLiStructure} gives
$d_s(z_s,q_s)\le C\sqrt h$ for small $h$.
Apply \cite[Theorem~5.30]{FangLiStructure} with $\epsilon=1/2$.
The triangle inequality and Young's inequality give
\[
 d_s(y,q_s)^2\le\frac{10}{9}d_s(y,z_s)^2
                         +10d_s(z_s,q_s)^2,
\]
so the Gaussian estimate becomes
\begin{equation}
 K(y)\le Ch^{-n/2}
       \exp\left(-\frac{d_s(y,q_s)^2}{5h}\right).
 \label{eq:terminal-gaussian}
\end{equation}

Write $D_s(y)=d_s(y,q_s)$. On a short interval ending at $t$,
\eqref{eq:slab-potential} and smoothness along $q_s$ give
$f(y)\le C(1+D_s(y)^2)$. Hence \eqref{eq:subcritical-growth}
implies $|u(y)|\le C\exp(CD_s(y)^2)$. For sufficiently small $h$,
\[
 (1+|u(y)|)K(y)\le Ch^{-n/2}
                    \exp\left(-\frac{D_s(y)^2}{6h}\right).
\]
The potential bound and \eqref{eq:sublevel-volume} also imply
\[
 \Vol_{g_s^{Z'}}B_{d_s}(q_s,L)\le C(1+L)^n\qquad(L>0).
\]
For fixed $r>0$, sum over the annuli
$2^kr\le D_s<2^{k+1}r$, $k\ge0$, to obtain
\[
 \int_{\{D_s\ge r\}}(1+|u|)\dd\nu_{q;s}
       \le C_rh^{-n/2}\exp\left(-\frac{r^2}{12h}\right)
       \longrightarrow0.
\]
Given $\eps>0$, continuity of $u$ in the regular cylinder provides
$r>0$ such that $|u(y)-u(q)|<\eps$ when $D_s(y)<r$ and $s$ is
sufficiently close to $t$. Since $\nu_{q;s}$ is a probability measure,
\[
 \left|\int_{\Reg'_s}u\dd\nu_{q;s}-u(q)\right|
 \le\eps+\int_{\{D_s\ge r\}}\bigl(|u|+|u(q)|\bigr)\dd\nu_{q;s}.
\]
Letting $s\nearrow t$ and then $\eps\searrow0$ proves
\begin{equation}
 \lim_{s\nearrow t}\int_{\Reg'_s}u\dd\nu_{q;s}=u(q).
 \label{eq:terminal-limit}
\end{equation}
The constant pairing and the two endpoint limits prove
Theorem~\ref{thm:representation}.

\begin{corollary}[Potential moment identity]
\label{cor:moment}
For $s<t<0$ and $q\in\Reg'_t$,
\[
 |t|\left(f(q)-\frac n2\right)
 =|s|\int_{\Reg'_s}\left(f(y)-\frac n2\right)\dd\nu_{q;s}(y).
\]
Equivalently,
\begin{equation}
 \int_{\Reg'_s}f(y)\dd\nu_{q;s}(y)
 =\frac{|t|}{|s|}f(q)
   +\left(1-\frac{|t|}{|s|}\right)\frac n2.
 \label{eq:moment-solved}
\end{equation}
\end{corollary}

\begin{proof}
By \eqref{eq:f-identities} and $\partial_{\Time'}\tau=-1$,
\[
 (\partial_{\Time'}-\Delta)\left[\tau\left(f-\frac n2\right)\right]
 =-\left(f-\frac n2\right)+\tau\frac{f-n/2}{\tau}=0.
\]
This function has at most quadratic growth in the slice distance on
every compact negative-time interval. Apply
Theorem~\ref{thm:representation} and rearrange.
\end{proof}

\section{Completion and distance comparison}
\label{sec:completion}

\subsection{Proof of the completion theorem}

\begin{lemma}
\label{lem:modified-time}
Let $J\Subset(-\infty,0)$ and $A<\infty$. Suppose $q_t$ is a
trajectory of $\partial_{\Time'}-\nabla f$ with $f(q_t)\le A$.
For $s,t\in J$ sufficiently close,
\begin{equation}
 d^*(q_s,q_t)\le C(n,Y,J,A)\sqrt{|t-s|}.
 \label{eq:modified-time-distance}
\end{equation}
\end{lemma}

\begin{proof}
Throughout the proof, $C$ depends only on $n,Y,J,A$ and may change
from line to line.
Assume $s<t$. Along the trajectory,
$f$ is constant and
$|\nabla f|^2+\mathrm{R}=f/\tau\le C$. The lower heat-kernel estimate
obtained from the reduced length of this regular curve gives
\begin{equation}
 K_{Z'}(q_t;q_s)\ge(4\pi(t-s))^{-n/2}e^{-C(t-s)}.
 \label{eq:curve-kernel-lower}
\end{equation}
This is the estimate used in
\cite[equation~(7.15)]{FangLiStructure}, obtained by applying
\cite[Corollary~9.4]{Perelman2002} on the closed approximating Ricci flows and
passing to the limit on the compact regular trajectory.
For an $H_n$-center $z_s$ of $q_t$, compare
\eqref{eq:curve-kernel-lower} with
\eqref{eq:gaussian-kernel}. This yields
$d_s(q_s,z_s)\le C\sqrt{t-s}$.
Equation~\eqref{eq:center} then implies
\[
 \Wone^{d_s}(\delta_{q_s},\nu_{q_t;s})
 \le d_s(q_s,z_s)+\sqrt{H_n(t-s)}\le C\sqrt{t-s}.
\]
Choosing the comparison time $s$ in \eqref{eq:dstar} proves the
claim.
\end{proof}

\begin{proof}[Proof of Theorem~\ref{thm:completion}]
Let $J\Subset(-\infty,0)$ be a compact interval. We prove that the
associated metric soliton restricted to $J$ is complete for $d^*$.
Let $(q_i)$ be a $d^*$-Cauchy sequence in that restriction.
By density we may replace it by a regular sequence at $d^*$-distance
tending to zero. Write $t_i=\Time'(q_i)\to t_0\in J$ and fix
$s<t_0$. For large $i$ all $t_i>s$.

The definition of $d^*$ and Wasserstein monotonicity imply
\[
 \Wone^{d_s}(\nu_{q_i;s},\nu_{q_j;s})\longrightarrow0.
\]
If $c_i$ is an $H_n$-center of $q_i$ at time $s$, then
\begin{align*}
 d_s(c_i,c_j)
 &\le\sqrt{H_n(t_i-s)}
   +\Wone^{d_s}(\nu_{q_i;s},\nu_{q_j;s})
   +\sqrt{H_n(t_j-s)}.
\end{align*}
Thus the centers remain in a bounded slice ball. The concentration
estimate gives the stronger conclusion
\[
 \int d_s(p_s,y)^2\dd\nu_{q_i;s}(y)
 \le2d_s(p_s,c_i)^2+2H_n(t_i-s)\le C.
\]

By \eqref{eq:quadratic-potential},
$\int f\dd\nu_{q_i;s}\le C$. Equation~\eqref{eq:moment-solved}
therefore yields
\begin{equation}
 \sup_i f(q_i)<\infty.
 \label{eq:cauchy-potential-bound}
\end{equation}
Transport $q_i$ by $\partial_{\Time'}-\nabla f$ to a point
$\widehat q_i$ at time $t_0$. The potential is unchanged, so the
lower bound in \eqref{eq:quadratic-potential} places
$\widehat q_i$ in a bounded ball of the associated metric soliton
slice. By the associated-soliton structure recalled in
Section~\ref{subsec:tangent-background}, this slice is a complete,
locally compact length space and hence proper. After passing to a subsequence,
\[
 \widehat q_i\longrightarrow q_\infty\quad\text{in }d_{t_0}.
\]
Lemma~\ref{lem:modified-time} and
\eqref{eq:cauchy-potential-bound} give
$d^*(q_i,\widehat q_i)\to0$, while
\eqref{eq:elementary-distances} gives
$d^*(\widehat q_i,q_\infty)\to0$. The original Cauchy sequence
thus converges in the associated metric soliton.

For a point of $Z'_{(-\infty,0)}$, approximate it by regular
spacetime points in a compact interval contained in $(-\infty,0)$.
The preceding completeness argument shows that its limit belongs
to the image of $\iota_{z'}$. This proves
\eqref{eq:surjectivity}. Since the associated metric soliton is
continuous, the distance-recovery identity used in the proof of
\cite[Lemma~6.5]{FangLiStructure} identifies its slice metrics with
the Fang--Li slice metrics in \eqref{eq:slice-distance-definition}.
The intrinsic
completion property of the associated soliton therefore identifies
the whole of
$(Z'_t,d_t^{Z'})$ with the completion of
$(\Reg'_t,d_{g_t^{Z'}})$.
\end{proof}

\begin{corollary}[Four-dimensional tangent Ricci shrinkers]
\label{cor:four-dimensional-orbifold}
Every four-dimensional tangent Ricci shrinker is a complete smooth
Riemannian orbifold with isolated singularities. The same holds for
every negative time slice of its tangent flow.
\end{corollary}

\begin{proof}
By Theorem~\ref{thm:completion}, the slice identifies isometrically
with the corresponding slice of the associated metric soliton, so
the conclusion follows from Bamler~\cite[Theorem~2.46]{BamlerStructure}.
\end{proof}

\subsection{Radial distance comparison}
\label{subsec:radial-comparison}

We identify the associated soliton with $Z'_{(-\infty,0)}$ and
write $d_t=d_t^{Z'}$. Fix a regular $H_n$-center
$p\in\Reg'_{-1}$ of $z'$ and set
$p_{-\tau}=\boldsymbol{\varphi}_\tau(p)$ for $\tau>0$. Define
\[
 D(q)=d_{Z'}(z',q),\qquad q\in Z'_{(-\infty,0)}.
\]
The reference centers satisfy
\begin{equation}
 \int d_{-a}(p_{-a},y)^2\dd\nu_{z';-a}(y)\le C(n)a
 \qquad(a>0).
 \label{eq:base-second-moment}
\end{equation}
This follows by scaling the center inequality at time $-1$.

\begin{theorem}[Distance comparison]
\label{thm:comparison}
There is $C=C(n)>1$ such that, for every $\tau>0$ and
$q\in Z'_{-\tau}$,
\begin{equation}
 C^{-1}\bigl(\sqrt\tau+d_{-\tau}^{Z'}(p_{-\tau},q)\bigr)
 \le d_{Z'}(z',q)
 \le C\bigl(\sqrt\tau+d_{-\tau}^{Z'}(p_{-\tau},q)\bigr).
 \label{eq:radial-comparison}
\end{equation}
The continuous extension of $f$ to $Z'_{-\tau}$ satisfies
\begin{equation}
 C^{-1}\frac{d_{Z'}(z',q)^2}{\tau}
 \le 1+f(q)\le C\frac{d_{Z'}(z',q)^2}{\tau}.
 \label{eq:potential-spacetime}
\end{equation}
In particular, for $r\ge1$,
\begin{equation}
 \begin{split}
 B_{d_{Z'}}(z',r)\cap Z'_{-1}
       &\subset B_{d_{-1}^{Z'}}(p,Cr),\\
 B_{d_{-1}^{Z'}}(p,r)
       &\subset B_{d_{Z'}}(z',Cr).
 \end{split}
 \label{eq:ball-comparison}
\end{equation}
\end{theorem}

\begin{proof}[Proof of Theorem~\ref{thm:comparison}]
First take $q\in\Reg'_{-\tau}$. The comparison time $-\tau$,
\eqref{eq:base-second-moment}, and the triangle inequality give
\begin{equation}
 D(q)\le d_{-\tau}(p_{-\tau},q)+C\sqrt\tau.
 \label{eq:radial-upper}
\end{equation}
Also $D(q)\ge\sqrt\tau$ by the time inequality.

Fix $L>D(q)$. There is a comparison time $-a$, with
$\tau\le a<L^2$, such that
\begin{equation}
 \Wone^{d_{-a}}(\nu_{z';-a},\nu_{q;-a})<L.
 \label{eq:optimizing-time}
\end{equation}
For $a>\tau$, let $c$ be an $H_n$-center of $q$ at time $-a$;
for $a=\tau$, set $c=q$ and $\nu_{q;-a}=\delta_q$. We have
\begin{align}
 d_{-a}(p_{-a},c)
 &\le C\sqrt a+L+\sqrt{H_n(a-\tau)}\le CL,
 \label{eq:comparison-center}\\
 \int d_{-a}(p_{-a},y)^2\dd\nu_{q;-a}(y)
 &\le2d_{-a}(p_{-a},c)^2+2H_n(a-\tau)\le CL^2.
 \label{eq:comparison-second-moment}
\end{align}
The upper potential estimate gives
\[
 \int f\dd\nu_{q;-a}\le C\left(1+\frac{L^2}{a}\right).
\]
On the other hand, \eqref{eq:moment-solved} gives
\[
 \frac\tau a f(q)
   =\int f\dd\nu_{q;-a}-\left(1-\frac\tau a\right)\frac n2
   \le C\left(1+\frac{L^2}{a}\right).
\]
Since $a<L^2$, we conclude that
$f(q)\le CL^2/\tau$. Using the lower potential estimate at time
$-\tau$, and then $L\searrow D(q)$, gives
\begin{equation}
 d_{-\tau}(p_{-\tau},q)\le C D(q).
 \label{eq:radial-lower}
\end{equation}
Together with \eqref{eq:radial-upper} and
$D\ge\sqrt\tau$, this proves
\eqref{eq:radial-comparison} on the regular set.

By \eqref{eq:f-identities}, $\sqrt f$ is
$1/(2\sqrt\tau)$-Lipschitz for the intrinsic slice distance and
extends continuously to the completion.
The function $D$ is also continuous there, since
$|D(q)-D(q')|\le d_{Z'}(q,q')\le d_{-\tau}(q,q')$.
Regular density thus proves \eqref{eq:radial-comparison} and the
quadratic potential bounds on the whole slice.
Combining these two comparisons yields
\eqref{eq:potential-spacetime}. Finally, setting $\tau=1$ proves
\eqref{eq:ball-comparison}, after increasing $C$.
\end{proof}

\begin{corollary}
\label{cor:proper-embedding}
For every $\tau>0$, the inclusion
\[
 (Z'_{-\tau},d_{-\tau}^{Z'})\longrightarrow(Z',d_{Z'})
\]
is proper and is a homeomorphism onto its image. A sequence on the
slice is intrinsically bounded if and only if its spacetime distance
from $z'$ is bounded. In particular, $Z'_{-\tau}$ is compact if and
only if it is bounded with respect to $d_{Z'}$.
\end{corollary}

\begin{proof}
The inclusion is $1$-Lipschitz. By
\eqref{eq:radial-comparison}, the inverse image of a compact
spacetime set is bounded in the slice metric. It is also closed
there, hence compact by properness of the slice. To see the inverse
continuity directly, a spacetime-convergent slice sequence is
intrinsically bounded, so every subsequence has an intrinsically
convergent subsubsequence. The $1$-Lipschitz inclusion identifies
its limit with the given spacetime limit. This proves continuity
of the inverse. The final assertion follows again from properness
and the radial comparison.
\end{proof}

\section{Spacetime curves and their projections}
\label{sec:action}

In Sections~\ref{sec:action}--\ref{sec:weak-applications}, write
\[
 (X,d)=(Z'_{-1},d_{-1}^{Z'}),\qquad g=g_{-1}^{Z'}.
\]
By Theorem~\ref{thm:completion}, $d=d_g$ on $\Reg'_{-1}$.
We retain $\tau=-\Time'>0$ and the soliton identifications
$\boldsymbol{\varphi}_\tau:\Reg'_{-1}\to\Reg'_{-\tau}$.
We construct regular minimizing geodesics by projecting spacetime
curves to $\Reg'_{-1}$.

\subsection{Reduction to ordinary energy}

Let a regular spacetime curve be written as
$\beta(\tau)=\boldsymbol{\varphi}_\tau(z(\tau))$, $a\le\tau\le b$, where
$0<a<b$. Thus $z(\tau)=\boldsymbol{\varphi}_\tau^{-1}(\beta(\tau))$
is its projection to $\Reg'_{-1}$. If
$\beta'=-\partial_{\Time'}+V$, then
\begin{equation}
 V=\nabla f+(\boldsymbol{\varphi}_\tau)_*z',\qquad
 |V|_{g_{-\tau}^{Z'}}^2
       =\tau\left|z'+\tau^{-1}\nabla f\right|_g^2.
 \label{eq:action-velocity}
\end{equation}
Define
\begin{equation}
 \Act_{a,b}(\beta)
       =\int_a^b\sqrt\tau\bigl(\mathrm{R}(\beta(\tau))+|V|^2\bigr)\dd\tau.
 \label{eq:absolute-action}
\end{equation}
This is Perelman's $\mathcal L$-length between two positive backward
times; see \cite[Section~7.1]{Perelman2002} and
\cite[Chapters~6--7]{MorganTian}. It is twice the
$\mathcal L_-$-length normalized in \cite[Definition~3]{Lott2009}.

\begin{proposition}[Reparametrization]
\label{prop:action-conjugacy}
Set
\[
 s(\tau)=2(a^{-1/2}-\tau^{-1/2}),\qquad
 S=s(b),\qquad \gamma(s)=z(\tau(s)).
\]
Then
\begin{equation}
 \Act_{a,b}(\beta)
 =2\sqrt b\,f(\gamma(S))-2\sqrt a\,f(\gamma(0))
       +\int_0^S|\dot\gamma|_g^2\dd s.
 \label{eq:action-conjugacy}
\end{equation}
\end{proposition}

\begin{proof}
Scalar curvature pulls back as $\tau^{-1}\mathrm{R}_g$. Therefore
\eqref{eq:action-velocity} and $\mathrm{R}_g+|\nabla f|_g^2=f$ give
\begin{align*}
 \sqrt\tau(\mathrm{R}+|V|^2)
 &=\tau^{3/2}|z'|_g^2+2\sqrt\tau\,\mathrm df(z')
      +\tau^{-1/2}(\mathrm{R}_g+|\nabla f|_g^2)\\
 &=\tau^{3/2}|z'|_g^2
      +\frac{\mathrm d}{\mathrm d\tau}
           \bigl(2\sqrt\tau\,f(z(\tau))\bigr).
\end{align*}
Since $\mathrm ds/\mathrm d\tau=\tau^{-3/2}$, the first term
integrates to $\int_0^S|\dot\gamma|_g^2\dd s$.
\end{proof}

A related pullback of Perelman's $\mathcal L$-functional to a fixed
smooth Ricci shrinker, followed by integration by parts in the
potential term, appears in Wang--Wang~\cite[Section~8]{WangWangHeat}.
Their backward time is measured from a smooth time slice, and the
resulting formula retains a potential integral. Here the backward-time
clock is anchored at time zero of the tangent flow, so the potential
terms combine into the endpoint contribution in
\eqref{eq:action-conjugacy}.

From now on we take
\begin{equation}
 a=1,\quad b=4,\quad s=2(1-\tau^{-1/2}),\quad
 \tau=(1-s/2)^{-2},\quad 0\le s\le1.
 \label{eq:one-four}
\end{equation}
For endpoints $x,y\in\Reg'_{-1}$, put
\[
 B(x,y)=4f(y)-2f(x),\qquad
 E_g(\gamma)=\int_0^1|\dot\gamma|_g^2\dd s.
\]
Thus $\Act_{1,4}=B+E_g$, and
\[
 \inf_{\beta\text{ regular}}\Act_{1,4}(\beta)
       =B(x,y)+d(x,y)^2.
\]
Indeed, Cauchy--Schwarz bounds energy from below by squared length,
and regular curves with lengths approaching the intrinsic distance
can be parametrized with constant speed. Our goal is to find a
regular curve which attains this infimum.

\subsection{The normal region on a smooth closed flow}

For the moment, let $(M,g(\tau))$, $1\le\tau\le4$, be a smooth
closed backward Ricci flow, so $\partial_\tau g=2\Ric$.
Fix an initial point $x$ at $\tau=1$, and define
\begin{equation}
 \mathcal L(z,\tau)=\inf_{\beta(1)=x,\,\beta(\tau)=z}
       \int_1^\tau\sqrt u\bigl(\mathrm{R}+|\beta'|_{g(u)}^2\bigr)\dd u,
 \qquad \ell(z,\tau)=\frac{\mathcal L(z,\tau)}{2\sqrt\tau}.
 \label{eq:action-value}
\end{equation}
Every endpoint admits a smooth minimizing $\mathcal L$-geodesic
\cite[Lemma~7.2]{MorganTian}. Writing $V=\beta'$, its equation is
\begin{equation}
 \nabla_\tau V+2\Ric^\sharp(V)+\frac{V}{2\tau}
       -\frac12\nabla \mathrm{R}=0;
 \label{eq:action-geodesic}
\end{equation}
see \cite[equation~(7.2)]{Perelman2002} and
\cite[Lemma~6.4]{MorganTian}.

Let $E_\tau(v)$ be the $\mathcal L$-exponential map with
$\beta(1)=x$ and $\beta'(1)=v$. Let $\mathcal U_4\subset T_xM$
be the open set of initial velocities whose $\mathcal L$-geodesics
are the unique, nonconjugate minimizers to their endpoints at time
$4$. The complement of $E_4(\mathcal U_4)$ is a closed set of
measure zero. For $1<\tau\le4$, the map $E_\tau$ is injective and
nonsingular on $\mathcal U_4$; see
\cite[Propositions~6.28 and~6.30, Corollary~6.67]{MorganTian}.
Define $J$ by
\begin{equation}
 E_\tau^*(\dV_{g(\tau)})=J(v,\tau)\dd v,
 \label{eq:action-jacobian}
\end{equation}
where $\dd v$ is induced by $g(1)$ on $T_xM$. The map
\begin{equation}
 \Xi:\mathcal U_4\times(1,4)\longrightarrow M\times(1,4),\qquad
 \Xi(v,\tau)=(E_\tau(v),\tau)
 \label{eq:normal-map}
\end{equation}
is a diffeomorphism onto an open normal region. There, endpoint
variation gives $V=\nabla\ell$
\cite[Corollary~6.29]{MorganTian}, and
\begin{equation}
 \partial_\tau\ell=\frac{\mathrm{R}}2-\frac{|\nabla\ell|^2}2
                         -\frac\ell{2\tau};
 \label{eq:Hamilton-Jacobi}
\end{equation}
see \cite[equation~(6.15)]{Lott2009}.

\subsection{The pointwise Jacobian identity}

Let $\mathcal D=\partial_\tau+\nabla\ell\cdot\nabla$ denote
differentiation along a minimizing $\mathcal L$-geodesic in the
normal region. We use the same symbol $J$ for the function in
\eqref{eq:action-jacobian} and its transport by $\Xi$. Thus
$\mathcal D\log J$ corresponds to
$\mathrm d\log J(v,\tau)/\mathrm d\tau$ at fixed $v$.
The standard logarithmic Jacobian formula is
\begin{equation}
 \mathcal D\log J=\Delta\ell+\mathrm{R};
 \label{eq:log-jacobian}
\end{equation}
see the Jacobian calculation in the proof of
\cite[Proposition~6.81]{MorganTian}, or the volume transport
identities \cite[equations~(6.8) and~(6.17)]{Lott2009}.

\begin{proposition}[Jacobian identity]
\label{prop:jacobian-identity}
Along each minimizing $\mathcal L$-geodesic in the normal region,
the function
\begin{equation}
 Z=-\log J+\ell+\frac n2\log\tau
 \label{eq:jacobian-Z}
\end{equation}
satisfies
\begin{equation}
 \mathcal D^2Z+\frac3{2\tau}\mathcal DZ
       =\left|\Ric+\nabla^2\ell-\frac{g(\tau)}{2\tau}\right|^2\ge0.
 \label{eq:jacobian-identity}
\end{equation}
In particular, $\tau^{3/2}\mathcal DZ$ is nondecreasing.
\end{proposition}

This is the local form of \cite[Corollary~7, equation~(6.26)]{Lott2009}.
Indeed, for a smooth density $\rho$ transported by $\nabla\ell$,
the product $\rho(E_\tau(v),\tau)J(v,\tau)$ is constant in $\tau$;
applying that identity to arbitrary such densities supported in a
normal tube gives \eqref{eq:jacobian-identity} pointwise.

\subsection{Uniform Jacobian comparison}

Return to the approximating Ricci flows $(M_i,g_i(t))$ and fix a
regular point $x\in\Reg'_{-1}$. We write $\Act_i$ for the
$\mathcal L$-length in \eqref{eq:absolute-action} with $a=1$, $b=4$,
and metric $g_i(-\tau)$, and add the subscript $i$ to the smooth-flow
objects defined above. Suppose the minimizing $\mathcal L$-geodesics
under consideration start at $x_i\to x$ smoothly and satisfy
\begin{equation}
 \Act_i(\beta_i)\le A_0.
 \label{eq:uniform-action-bound}
\end{equation}
By \eqref{eq:scalar-lower}, after increasing $i$ we have
\begin{equation}
 \mathrm{R}_{g_i(-\tau)}\ge-1\qquad(1\le\tau\le4).
 \label{eq:scalar-minus-one}
\end{equation}
Together with \eqref{eq:uniform-action-bound}, this gives
\begin{equation}
 \int_1^4\sqrt\tau|V_i|^2\dd\tau\le E_0:=A_0+\frac{14}3.
 \label{eq:physical-energy}
\end{equation}

The initial localization follows the argument in
\cite[Lemma~3.22, equations~(3.84)--(3.85)]{ChenWangB}. We also
record the uniform initial Jacobian estimate needed below.

\begin{lemma}[Initial regular cap]
\label{lem:initial-cap}
There are $h>0$, $\theta\in(0,h)$, and $V_0,C_0<\infty$, independent
of large $i$ and of the minimizing curve, such that these curves
remain in a fixed compact regular coordinate cap on $[1,1+h]$,
$|V_i(1)|\le V_0$, and
\begin{equation}
 \mathcal DZ_i(v,1+\theta)\ge-C_0.
 \label{eq:initial-Z}
\end{equation}
\end{lemma}

\begin{proof}
Choose a relatively compact regular spacetime neighborhood about
$x$. In its time-preserving coordinates the pulled-back metrics
are uniformly equivalent and the coordinate drift is uniformly
bounded. Before a first exit, displacement on $[1,1+h]$ is bounded
by
\[
 C\int_1^{1+h}|V_i|\dd\tau+Ch\le C\sqrt{hE_0}+Ch.
\]
For sufficiently small $h$, this is less than the distance to the
lateral boundary. Thus an exit cannot occur in that interval.
The regular cap has uniform bounds for $\Ric_i$ and $\nabla \mathrm{R}_i$.
Equation~\eqref{eq:action-geodesic} gives
\[
 \frac{\mathrm d}{\mathrm d\tau}|V_i|^2
 =-2\Ric_i(V_i,V_i)+\langle\nabla \mathrm{R}_i,V_i\rangle
                         -\frac{|V_i|^2}\tau,
\]
so the absolute derivative is bounded by $C(1+|V_i|^2)$.
Some time in this interval has $|V_i|^2\le E_0/h$.
A simple ODE estimate then bounds the speed throughout the cap.

We give the corresponding Jacobi-field estimate. In regular local
frames along the curve, let $A_i(\tau)$ have columns given by the
components of $Y_{i,a}(\tau)=\mathrm dE_{i,\tau}(e_a)$, where
$\{e_a\}$ is $g_i(-1)$-orthonormal and the frames agree with it at
$\tau=1$. The $\mathcal L$-Jacobi equation has the form
\[
 A_i''+B_iA_i'+C_iA_i=0,\qquad A_i(1)=0,\qquad A_i'(1)=\Id;
\]
see \cite[Lemma~6.10, equation~(6.6), and Lemma~6.19]{MorganTian}.
Primes denote differentiation in $\tau$. Smooth convergence on the
fixed regular cap bounds the metric coefficients and their
required derivatives. Together with the speed bound, this gives
uniform bounds for $B_i$ and $C_i$, independent of $i$ and of the
minimizing curve. The ODE first bounds $A_i$ and $A_i'$ uniformly.
Integrating
\[
 A_i'(1+q)-\Id
       =-\int_1^{1+q}(B_iA_i'+C_iA_i)\dd\tau
\]
and then integrating once more gives, uniformly as $q\searrow0$,
\[
 A_i(1+q)=q\Id+O(q^2),\qquad A_i'(1+q)=\Id+O(q).
\]
Let $G_i(\tau)$ be the metric matrix in these frames. Then
$G_i(1)=\Id$, and $G_i$, $G_i^{-1}$, and $G_i'$ are uniformly
bounded. For sufficiently small $q>0$, $\det A_i(1+q)>0$, and
\[
 J_i=\det A_i\sqrt{\det G_i},\qquad
 \mathcal D\log J_i
   =\operatorname{tr}(A_i'A_i^{-1})
       +\frac12\operatorname{tr}(G_i^{-1}G_i')
   =\frac nq+O(1).
\]
Choose a fixed sufficiently small $\theta\in(0,h)$. The bounded
speed and scalar curvature give $\ell_i(1+\theta)=O(\theta)$,
while \eqref{eq:Hamilton-Jacobi} gives
$\mathcal D\ell_i(1+\theta)=O(1)$. Substitution in
\eqref{eq:jacobian-Z} yields
$\mathcal DZ_i(1+\theta)=-n/\theta+O(1)\ge-C_0$,
as required.
\end{proof}

\begin{proposition}[Endpoint-to-interior Jacobian comparison]
\label{prop:jacobian-comparison}
There is $C_J<\infty$ such that every minimizing
$\mathcal L$-geodesic in the normal region satisfying \eqref{eq:uniform-action-bound} obeys
\begin{equation}
 J_i(v,4)\le C_J J_i(v,\tau)
       \qquad(1+\theta\le\tau<4).
 \label{eq:jacobian-comparison}
\end{equation}
The constant depends only on $n$, $A_0$, $\theta$, and the uniform
geometry of the initial cap.
\end{proposition}

\begin{proof}
Every initial segment of a minimizing $\mathcal L$-geodesic is
minimizing. By \eqref{eq:scalar-minus-one},
\[
 -\int_1^\tau\sqrt u\dd u
 \le\mathcal L_i(E_{i,\tau}(v),\tau)
 \le A_0+\int_\tau^4\sqrt u\dd u.
\]
Thus $|\ell_i(E_{i,\tau}(v),\tau)|\le C(A_0)$ on $[1,4]$.
Proposition~\ref{prop:jacobian-identity} and
Lemma~\ref{lem:initial-cap} imply
$\tau^{3/2}\mathcal DZ_i(\tau)\ge-C$ for
$\tau\ge1+\theta$. Integration gives
$Z_i(4)-Z_i(\tau)\ge-C$. Hence
\[
 \log\frac{J_i(v,4)}{J_i(v,\tau)}
 =-[Z_i(4)-Z_i(\tau)]
      +\ell_i(4)-\ell_i(\tau)+\frac n2\log\frac4\tau\le C.
\]
\end{proof}

\section{Last entry and weak convexity}
\label{sec:weak}

We prove the following weak-convexity statement.

\begin{theorem}[Weak convexity]
\label{thm:weak}
Set $g=g_{-1}^{Z'}$. For every $x\in\Reg'_{-1}$,
there is a Borel null set $N_x\subset\Reg'_{-1}$ such that each
$y\in\Reg'_{-1}\setminus N_x$ is joined to $x$ by a unique regular
minimizing geodesic of length
\[
 d_g(x,y)=d_{-1}^{Z'}(x,y).
\]
Here null sets are measured by $\dV_g$.
\end{theorem}

We obtain it from the quantitative compact-containment estimate below.
Fix $x\in\Reg'_{-1}$, a compact set $K\Subset\Reg'_{-1}$, and
$\eps\in(0,1)$. All constants in this section may depend on
$x$, $K$, $\eps$, and the fixed tangent flow.
For a compact set $D\Subset\Reg'_{-1}$ containing $x$, define
\[
 G_x(D)=\{y\in D:\text{a minimizing $g$-geodesic from $x$ to $y$
 has image in $D$}\}.
\]
Minimality here is for the intrinsic distance on all of $\Reg'_{-1}$.

\begin{theorem}[Quantitative compact containment]
\label{thm:quantitative}
There are $L,C<\infty$ and $\xi_0>0$ such that, for every
$0<\xi<\xi_0$, there is a compact set $D_\xi\Subset\Reg'_{-1}$ with
\begin{equation}
 \{x\}\cup K\subset\operatorname{int}_{\Reg'_{-1}}D_\xi,
 \qquad D_\xi\subset\{2\sqrt f\le L+1\}\cap\{\rRm\ge\xi\},
 \label{eq:Dxi-containment}
\end{equation}
and
\begin{equation}
 \Vol_g\bigl(K\setminus G_x(D_\xi)\bigr)\le C\xi^{3-\eps}.
 \label{eq:quantitative-attainment}
\end{equation}
\end{theorem}

\subsection{Domains with small inner boundary}

Put $r(z)=d(p,z)$ and
\begin{equation}
 \rho=2\sqrt f,\qquad
 |\nabla\rho|^2=1-\frac{\mathrm R}{f}\le1.
 \label{eq:potential-radius}
\end{equation}
The function $\rho$ is smooth on $\Reg'_{-1}$. The coarse bounds
\eqref{eq:quadratic-potential} imply
$\{\rho\le L+2\}\subset B_g(p,C(L+3))$ for $L\ge1$.
Let $L$ be a large number to be fixed below, independently of $\xi$,
and let $\eta_\xi$ be the restriction to $\Reg'_{-1}$ of the cutoff
in Lemma~\ref{lem:cutoff-package}. By \eqref{eq:slice-bad-volume},
coarea gives
\begin{align*}
 \int_{1/3}^{2/3}\Area_g\bigl(\{\eta_\xi=c\}\cap\{\rho<L+2\}\bigr)\dd c
 &\le\int_{\{\rho<L+2\}\cap\{0<\eta_\xi<1\}}
                    |\nabla\eta_\xi|\dd V_g\\
 &\le C_{L,\eps}\xi^{3-\eps}.
\end{align*}
Choose a regular value $c_\xi\in(1/3,2/3)$ with
\begin{equation}
 \Sigma_\xi=\{\eta_\xi=c_\xi\},\qquad
 \Area_g\bigl(\Sigma_\xi\cap\{\rho<L+2\}\bigr)
          \le C_{L,\eps}\xi^{3-\eps}.
 \label{eq:small-area}
\end{equation}
Choose $L_\xi\in(L,L+1)$ which is a regular value both of $\rho$
and of its restriction to $\Sigma_\xi$. Set
\[
 \Omega_\xi=\{\eta_\xi>c_\xi\}\cap\{\rho<L_\xi\},\qquad
 D_\xi=\overline{\Omega_\xi}^{\Reg'_{-1}}.
\]
The regular cutoff compactness makes $D_\xi$ compact in $\Reg'_{-1}$,
and it is contained in $\{\rho\le L+1\}\cap\{\rRm\ge\xi\}$.
In particular, $D_\xi\subset B_g(p,C(L+2))$.
There are two boundary faces,
\[
 \partial_{\rm in}\Omega_\xi=\Sigma_\xi\cap\{\rho\le L_\xi\},\qquad
 \partial_{\rm out}\Omega_\xi=\{\rho=L_\xi\}\cap\{\eta_\xi\ge c_\xi\}.
\]
Only the inner face has small area. We will exclude the outer face
as a possible last-entry location.

\subsection{Moving coordinates and a uniform length bound}

Let $\phi_i$ be the smooth spacetime convergence maps from
Section~\ref{subsec:tangent-background}, and put
\[
 \Psi_i(z,\tau)=\phi_i(\boldsymbol{\varphi}_\tau(z)).
\]
We regard $\Psi_i$ as its spatial component in the fixed manifold
$M_i$, with time coordinate $-\tau$. Define
\[
 \begin{split}
 h_i(\tau)&=\Psi_i(\cdot,\tau)^*g_i(-\tau),\\
 w_i&=(\mathrm d_z\Psi_i)^{-1}\partial_\tau\Psi_i,\qquad
 \mathrm{R}_i=\mathrm{R}_{g_i(-\tau)}\circ\Psi_i.
 \end{split}
\]
Convergence of the time vector fields, as well as the metrics,
gives
\begin{equation}
 h_i\longrightarrow\tau g,\qquad
 w_i\longrightarrow\tau^{-1}\nabla f,\qquad
 \mathrm{R}_i\longrightarrow\tau^{-1}\mathrm{R}_g
 \label{eq:moving-convergence}
\end{equation}
smoothly on $A\times[1,4]$ for every compact
$A\Subset\Reg'_{-1}$. For every fixed $\xi$
and sufficiently large $i$ we may demand
\begin{equation}
 \frac{1}{2}\tau g\le h_i\le2\tau g,\qquad
 |w_i|_g\le C(1+r)
 \label{eq:moving-comparison}
\end{equation}
on a neighborhood of $D_\xi\times[1,4]$. The constant $C$ in the
drift estimate is independent of $\xi$ and $L$, by
\eqref{eq:linear-drift}; only the required index can depend on them.
Set
\[
 U_i=\partial_\tau\Psi_i=(\mathrm d_z\Psi_i)w_i.
\]
For a curve $\beta_i(\tau)=\Psi_i(z_i(\tau),\tau)$,
\begin{equation}
 V_i=U_i+(\mathrm d_z\Psi_i)z_i',\qquad
 |V_i|_{g_i(-\tau)}=|z_i'+w_i|_{h_i}.
 \label{eq:relative-velocity}
\end{equation}

There is a uniform bound $A_0$ for the minimum $\mathcal L$-lengths from
$x_i=\Psi_i(x,1)$ to $y_i=\Psi_i(y,4)$, $y\in K$.
To construct competitors, cover $K$ by finitely many relatively
compact normal neighborhoods. Fix a regular path from $x$ to the
center of each neighborhood, append a short local path to $y$,
and parametrize the pieces on fixed subintervals of $[0,1]$.
These curves have uniformly bounded energy and lie in a fixed
compact regular set. Their soliton lifts and their images under
$\phi_i$ have uniformly bounded $\mathcal L$-lengths. Global minimizing
$\mathcal L$-geodesics
on the closed approximating Ricci flows have no greater $\mathcal L$-length. In particular,
\eqref{eq:physical-energy} applies with fixed $A_0,E_0$.

The fixed endpoint neighborhoods can be included in $\Omega_\xi$
by first choosing $L$ large and then $\xi$ small. Their images under
$\Psi_i$ contain the initial regular cap of
Lemma~\ref{lem:initial-cap}.

\subsection{Localization after the last entry}

For $A\subset\Reg'_{-1}$, define its moving domain in backward-time
coordinates by
\[
 \mathcal T_i(A)
   =\{(\Psi_i(z,\tau),\tau):z\in A,\ 1\le\tau\le4\}.
\]
A minimizing $\mathcal L$-geodesic from $x_i$ to
$\Psi_i(y,4)$, $y\in K$, is called \emph{bad} if its spacetime graph
is not contained in $\mathcal T_i(\Omega_\xi)$. Since its endpoints
lie in the interior spatial sections, a bad curve has a last
boundary contact at some $\bar\tau\in(1,4)$. For
$\bar\tau<\tau\le4$, it remains in $\mathcal T_i(\Omega_\xi)$,
and we write it in the coordinates $z_i(\tau)$.

We follow the localization argument in
\cite[proof of Lemma~3.22, equations~(3.84)--(3.89)]{ChenWangB},
using moving soliton coordinates.

\begin{lemma}[Terminal localization and endpoint separation]
\label{lem:last-entry-localization}
There are $R_*,\delta_a,\delta_b>0$, independent of $\xi$, such that
for sufficiently large $i$ after $\xi$ is fixed,
\begin{equation}
 r(z_i(\tau))\le R_*\quad(\bar\tau\le\tau\le4),\qquad
 1+\delta_a\le\bar\tau\le4-\delta_b.
 \label{eq:last-entry-localization}
\end{equation}
The outer radius $L$ can be fixed so that a last contact on
$\partial_{\rm out}\Omega_\xi$ is impossible. The initial
$\theta$ in \eqref{eq:jacobian-comparison} can be chosen with
$\theta<\delta_a$.
\end{lemma}

\begin{proof}
By \eqref{eq:moving-comparison} and
\eqref{eq:relative-velocity},
\begin{equation}
 |z_i'|_g\le C|V_i|_{g_i(-\tau)}+C(1+r(z_i)).
 \label{eq:coordinate-speed}
\end{equation}
The function $r$ is intrinsically $1$-Lipschitz. Integrating backward
from $z_i(4)=y\in K$ gives
\[
 r(z_i(\tau))\le\sup_Kr+C\int_\tau^4|V_i|\dd u
                         +C\int_\tau^4(1+r(z_i(u)))\dd u.
\]
The integral of $|V_i|$ is at most $\sqrt{3E_0}$.
Backward Gr\"onwall yields, for example,
\[
 R_*=e^{3C}\bigl(\sup_Kr+C\sqrt{3E_0}+3C\bigr).
\]
This constant is independent of the outer radius. Now choose
$L>\rho(p)+R_*+1$, increasing it to contain the fixed comparison
paths and endpoint neighborhoods in $\{\rho<L\}$.
Since \eqref{eq:potential-radius} gives $\rho\le\rho(p)+r$,
an outer-face contact would imply
$L< L_\xi=\rho\le\rho(p)+R_*$, a contradiction.
This also excludes the corner where the two faces meet.

The initial regular cap remains inside the moving domain for a
fixed interval. It gives the lower bound
$\bar\tau\ge1+\delta_a$ and allows $\theta<\delta_a$.
Choose $\bar d>0$ so that the intrinsic $\bar d$-neighborhood of $K$
has compact closure in $\Reg'_{-1}$ and is contained in $\Omega_\xi$
for all sufficiently small $\xi$. The last-contact point is at
distance at least $\bar d$ from $K$. On the terminal segment,
\eqref{eq:coordinate-speed} and the already proved radial bound give
\[
 \bar d\le\int_{\bar\tau}^4|z_i'|_g\dd\tau
       \le C\sqrt{E_0(4-\bar\tau)}+C(4-\bar\tau).
\]
The right side is smaller than $\bar d$ if $4-\bar\tau$ is
sufficiently small. This proves the fixed terminal separation.
\end{proof}

For $s$ as in \eqref{eq:one-four}, put
$\bar s=s(\bar\tau)$ and $\gamma_i(s)=z_i(\tau(s))$ on the
terminal segment. Equations~\eqref{eq:physical-energy} and
\eqref{eq:coordinate-speed} imply
\begin{equation}
 \int_{\bar s}^1|\dot\gamma_i|_g^2\dd s\le E_1,
 \qquad 1-\bar s\ge\frac{\delta_b}{8},
 \label{eq:terminal-integral-energy}
\end{equation}
where $E_1$ is independent of $\xi$. Indeed,
$\int|\dot\gamma_i|^2\dd s=\int\tau^{3/2}|z_i'|^2\dd\tau$,
and $\tau^{-3/2}\ge1/8$ on $[1,4]$.

\subsection{Uniform speed of projected curves}

We first express the $\mathcal L$-length as a functional of the
projected curve. For $0\le s_0<s_1\le1$, put $\tau_j=\tau(s_j)$
and $\beta_i(\tau)=\Psi_i(\gamma(s(\tau)),\tau)$. Write
$\Act_{i;\tau_0,\tau_1}$ for the $\mathcal L$-length on
$[\tau_0,\tau_1]$ with metric $g_i(-\tau)$, and define
\[
 \begin{split}
 \widehat{\mathcal A}_{i;s_0,s_1}(\gamma)
  &:=\Act_{i;\tau_0,\tau_1}(\beta_i)
       -\left[2\sqrt{\tau(s)}f(\gamma(s))\right]_{s_0}^{s_1}\\
  &=\int_{s_0}^{s_1}\widehat L_i(\gamma,\dot\gamma,s)\dd s.
 \end{split}
\]
In particular, for $\gamma(0)=x$ and $\gamma(1)=y$,
\[
 \widehat{\mathcal A}_{i;0,1}(\gamma)=\Act_i(\beta_i)-B(x,y).
\]
The endpoint term is fixed under variations with fixed endpoints.
By \eqref{eq:relative-velocity} and $\mathrm d\tau/\mathrm ds=\tau^{3/2}$,
the integrand is
\[
 \widehat L_i(z,v,s)=a_i(z,s)(v,v)+2b_i(z,s)(v)+c_i(z,s),
\]
with
\begin{equation}
 \begin{aligned}
 a_i&=\tau^{-1}h_i,\\
 b_i&=\sqrt\tau\bigl(h_i(w_i,\cdot)-\mathrm df\bigr),\\
 c_i&=\tau^2\bigl(\mathrm{R}_i+h_i(w_i,w_i)\bigr)-\tau f.
 \end{aligned}
 \label{eq:lagrangian-coefficients}
\end{equation}
For example, before subtracting the total derivative, the cross term
is $2\sqrt\tau h_i(w_i,v)$, whereas
\[
 \frac{\mathrm d}{\mathrm ds}(2\sqrt\tau f)=\tau f
                                  +2\sqrt\tau\,\mathrm df(v).
\]
By \eqref{eq:moving-convergence},
\begin{equation}
 a_i\longrightarrow g,\qquad b_i\longrightarrow0,\qquad
 c_i\longrightarrow0
 \label{eq:lagrangian-convergence}
\end{equation}
smoothly on $A\times[0,1]$ for each compact
$A\Subset\Reg'_{-1}$. For fixed $\xi$ there
are $\epsilon_{i,\xi}\to0$ such that
\begin{equation}
 \begin{gathered}
 \frac{1}{2}g\le a_i\le2g,\qquad |c_i|\le\epsilon_{i,\xi},\\
 |\partial_sa_i|_g+|\partial_sb_i|_g+|\partial_sc_i|
                           \le\epsilon_{i,\xi}.
 \end{gathered}
 \label{eq:coefficient-time-control}
\end{equation}
All norms here are taken with respect to $g$.

\begin{lemma}[Uniform relative speed]
\label{lem:relative-speed}
There is $C_V<\infty$, independent of $\xi$, such that every bad
minimizing curve satisfies
\begin{equation}
 |\dot\gamma_i|_g\le C_V\quad(\bar s\le s\le1),\qquad
 |V_i-U_i|_{g_i(-\bar\tau)}\le C_V'
 \label{eq:relative-speed}
\end{equation}
at the last-entry point, for all sufficiently large $i$ after
$\xi$ is fixed.
\end{lemma}

\begin{proof}
The terminal segment minimizes
$\widehat{\mathcal A}_{i;\bar s,1}$ with its endpoints fixed and
therefore satisfies the Euler--Lagrange equation. For a tangent vector $w$,
\[
 (\partial_v\widehat L_i)[w]=2a_i(v,w)+2b_i(w).
\]
Define the positive energy
\[
 e_i=a_i(v,v)-c_i+2
       =(\partial_v\widehat L_i)[v]-\widehat L_i+2,
 \qquad v=\dot\gamma_i.
\]
The linear term cancels. Along an Euler--Lagrange curve,
\[
 \frac{\mathrm d}{\mathrm ds}
       \bigl(v^\alpha L_{v^\alpha}-L\bigr)
 =v^\alpha\left(\frac{\mathrm d}{\mathrm ds}L_{v^\alpha}
                              -L_{z^\alpha}\right)-L_s
 =-\partial_sL.
\]
Applying this identity to $\widehat L_i$ gives
\[
 e_i'=-(\partial_sa_i)(v,v)-2(\partial_sb_i)(v)-\partial_sc_i,
\]
where $\partial_s$ is taken at fixed $(z,v)$.
For $i$ large enough that $\epsilon_{i,\xi}\le1$,
\eqref{eq:coefficient-time-control} yields
\[
 1+\frac{1}{2}|v|^2\le e_i\le3+2|v|^2,
 \qquad |e_i'|\le C\epsilon_{i,\xi}e_i\le Ce_i.
\]
For any $s,t$ on the terminal segment, $e_i(s)\le e^Ce_i(t)$.
Integrating in $t$ and using
\eqref{eq:terminal-integral-energy},
\[
 e_i(s)\le\frac{e^C}{1-\bar s}\int_{\bar s}^1e_i(t)\dd t
       \le\frac{8e^C}{\delta_b}(3+2E_1).
\]
This proves the first bound. At the contact point,
\[
 |V_i-U_i|=|z_i'|_{h_i}
       \le\sqrt{2\tau}\,\tau^{-3/2}|\dot\gamma_i|_g
       \le\sqrt2 C_V.
\]
All quantities extend continuously to the contact point in the
slightly larger regular coordinate domain.
\end{proof}

\subsection{The moving-boundary area formula}

Let $(M,g(\tau))$, $1\le\tau\le4$, be a smooth closed backward
Ricci flow satisfying $\partial_\tau g=2\Ric$. On the normal region
\eqref{eq:normal-map}, define the endpoint projection
\[
 P(E_\tau(v),\tau)=E_4(v).
\]
Let $\iota_\tau:N\to M$ be a smooth family of hypersurface
embeddings and set
\[
 S_\tau=\iota_\tau(N),\qquad
 U(\iota_\tau(q),\tau)=\partial_\tau\iota_\tau(q),\qquad
 \mathcal S=\{(\iota_\tau(q),\tau):q\in N,\ 1<\tau<4\}.
\]
Write $\nu$ for a local unit normal to $S_\tau$ and $\dd\sigma_\tau$
for its area density. Equip the spacetime track $\mathcal S$ with
the density $\dd\sigma_\tau\dd\tau$. The Jacobians below are relative
to this density and the endpoint volume density $\dV_{g(4)}$.

\begin{lemma}[Projection Jacobian]
\label{lem:flux}
On the part of this track contained in the normal region, the
absolute $n$-dimensional Jacobian of the restricted projection is
\begin{equation}
 \Jac P=\frac{J(v,4)}{J(v,\tau)}
                     |\langle V-U,\nu\rangle_{g(\tau)}|.
 \label{eq:flux-jacobian}
\end{equation}
Consequently, for a measurable subset $H$ of the track contained
in the normal region,
\begin{equation}
 \Vol_{g(4)}(P(H))
 \le\int_H\frac{J(v,4)}{J(v,\tau)}
              |\langle V-U,\nu\rangle|\dd\sigma_\tau\dd\tau.
 \label{eq:flux-inequality}
\end{equation}
\end{lemma}

\begin{proof}
Work in oriented local charts and let $\omega_\tau$ be the spatial
volume form, extended to spacetime by annihilating $\partial_\tau$.
Put
\[
 \Omega=\omega_\tau\wedge\mathrm d\tau,\qquad
 \mathcal X=\partial_\tau+V,\qquad
 \alpha=\iota_{\mathcal X}\Omega.
\]
Since $P\circ\Xi(v,\tau)=E_4(v)$, we have
$\mathrm dP(\mathcal X)=0$. Moreover,
$\mathrm d\Xi(\partial_\tau)=\mathcal X$ and
$\alpha(\mathcal X,\cdot)=0$. Choose a fixed orthonormal basis
$e_a$ in the initial-velocity space. Evaluating on the spatial frame
$Y_a=\mathrm dE_\tau(e_a)$ therefore gives the full pullback identity
\[
 \Xi^*\alpha=\pm J(v,\tau)\dd v,
 \qquad
 \Xi^*(P^*\omega_4)=E_4^*\omega_4=J(v,4)\dd v.
\]
Here local orientations can be chosen consistently; the sign of the
first identity will disappear on taking absolute densities.
Since $\Xi$ is a diffeomorphism, comparison of these identities yields
\begin{equation}
 P^*\omega_4
   =\pm\frac{J(v,4)}{J(v,\tau)}
       \iota_{\partial_\tau+V}
                   (\omega_\tau\wedge\mathrm d\tau)
 \label{eq:flux-form-identity}
\end{equation}
on the normal region, with $v$ understood through $\Xi^{-1}$.

At a point of the moving track, choose a spatial orthonormal
tangential frame $e_1,\ldots,e_{n-1}$ and let
$T=\partial_\tau+U$ be its time tangent. Since
$\mathcal X=T+(V-U)$, alternation gives
\begin{align*}
 |\alpha(e_1,\ldots,e_{n-1},T)|
 &=|\Omega(V-U,e_1,\ldots,e_{n-1},T)|\\
 &=|\omega_\tau(V-U,e_1,\ldots,e_{n-1})|\\
 &=|\langle V-U,\nu\rangle_{g(\tau)}|.
\end{align*}
The chosen frame has unit density for
$\dd\sigma_\tau\dd\tau$. Restricting
\eqref{eq:flux-form-identity} to it proves
\eqref{eq:flux-jacobian}. Only the normal component of $U$ matters,
so the formula does not depend on the tangential parametrization
of the moving hypersurface.

Apply the area formula in relatively compact coordinate patches.
If $N(P,H,y)$ denotes the number of preimages of $y$ in $H$, it gives
\[
 \int_H\Jac P\,\dd\sigma_\tau\dd\tau
   =\int_M N(P,H,y)\dV_{g(4)}(y)
   \ge\Vol_{g(4)}(P(H)).
\]
Summing over a countable coordinate decomposition proves
\eqref{eq:flux-inequality}.
\end{proof}

Let $\mathcal C_{i,4}$ be the closed null $\mathcal L$-cut set at
time $4$ on the $i$-th approximating Ricci flow. Let $B_{i,\xi}\subset K$
consist of its pullback under $\Psi_i(\cdot,4)$, together with the
remaining endpoints whose unique minimizing $\mathcal L$-geodesic
is bad. For each bad endpoint outside the cut set, let
$\bar\tau_i(y)$ be the last boundary-contact time, and denote the
image of the last-contact map
\[
 y\longmapsto
       \bigl(\beta_{i,y}(\bar\tau_i(y)),\bar\tau_i(y)\bigr)
\]
by $H_{i,\xi}$. Figure~\ref{fig:last-entry-projection} illustrates
the endpoint projection.

\begin{figure}[!htb]
 \centering
 \includegraphics[width=0.84\textwidth]{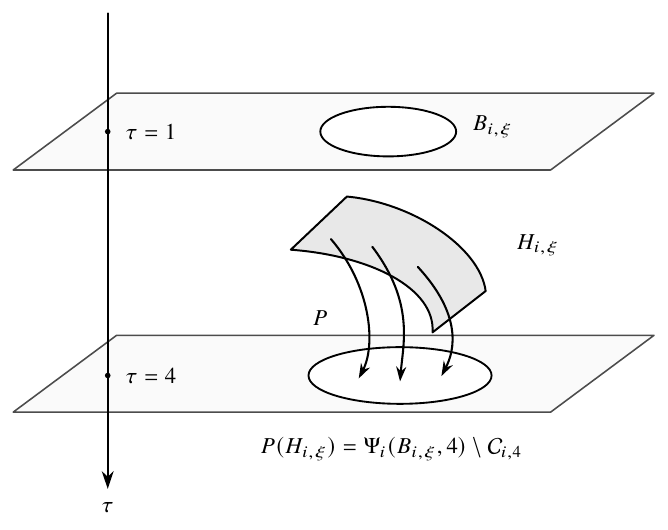}
 \caption{The reference set $B_{i,\xi}$ on the slice $\tau=1$,
 the last-contact set $H_{i,\xi}$, and the endpoint projection $P$
 to the slice $\tau=4$.}
 \label{fig:last-entry-projection}
\end{figure}

\begin{lemma}[Measurability]
\label{lem:measurability}
The set $B_{i,\xi}$, the last-contact time and map, and the set
$H_{i,\xi}$ are Borel. Moreover, $H_{i,\xi}$ lies in the normal
region on the smooth track of $\partial_{\rm in}\Omega_\xi$, at
times in $[1+\delta_a,4-\delta_b]$, and its endpoint projection
is exactly the image under $\Psi_i(\cdot,4)$ of the bad endpoints
outside the cut set. For every compact $D\Subset\Reg'_{-1}$ containing $x$, the set
$G_x(D)$ is closed in $D$.
\end{lemma}

\begin{proof}
Off the cut set, the minimizing curve depends smoothly on its
endpoint. Containment of its compact graph in the relatively open
set $\mathcal T_i(\Omega_\xi)$ is an open condition. Thus $B_{i,\xi}$
is Borel.

Lemma~\ref{lem:last-entry-localization} places all last contacts
on the inner boundary, away from both endpoint times and the outer
face, including its intersection with the inner face. The contacts
therefore lie on a smooth hypersurface track. They lie in the
normal region because the endpoint is outside the cut set.

For a fixed $a$, the condition $\bar\tau_i(y)\ge a$ means that
the curve meets the closed boundary track at some time in
$[\max\{a,1+\delta_a\},4-\delta_b]$. This condition is relatively
closed on the set of bad endpoints outside the cut set: along a
convergent sequence of endpoints, pass to a convergent subsequence
of contact times. Hence $\bar\tau_i$ and the last-contact map are
Borel. The latter is injective, since $P$ recovers $\Psi_i(y,4)$.
The Borel injective-image theorem implies that $H_{i,\xi}$ is
Borel, and the assertion about its projection follows.

Finally, if $y_j\in G_x(D)$ and $y_j\to y\in D$, their
constant-speed minimizing geodesics stay in $D$ and have speeds
$d(x,y_j)\to d(x,y)$. Standard compactness and lower
semicontinuity give a curve in $D$ from $x$ to $y$ with energy
at most $d(x,y)^2$. It minimizes energy in the whole regular
manifold and is therefore a smooth minimizing geodesic. Thus
$y\in G_x(D)$.
\end{proof}

\begin{proposition}[Avoidance estimate]
\label{prop:avoidance}
For $0<\xi<\xi_0$ and $i\ge i_0(\xi)$,
\begin{equation}
 \Vol_g(B_{i,\xi})\le C\xi^{3-\eps},
 \label{eq:avoidance}
\end{equation}
where $C$ is independent of $i$ and $\xi$.
\end{proposition}

\begin{proof}
On $H_{i,\xi}$, Proposition~\ref{prop:jacobian-comparison}
bounds $J_i(4)/J_i(\tau)$, and
Lemma~\ref{lem:relative-speed} bounds the relative normal speed.
The endpoint projection of that set is the set of bad endpoints
outside the null cut set. Hence \eqref{eq:flux-inequality} gives
\begin{equation}
 \Vol_{g_i(-4)}(\Psi_i(B_{i,\xi},4))
 \le C\int_{1+\delta_a}^{4-\delta_b}
       \Area_{g_i(-\tau)}(\Psi_i(\partial_{\rm in}\Omega_\xi,\tau))
                                                       \dd\tau.
 \label{eq:bad-volume-flux}
\end{equation}
On this track, $h_i\le2\tau g$, so each area is at most
$(2\tau)^{(n-1)/2}\Area_g(\partial_{\rm in}\Omega_\xi)$.
Equation~\eqref{eq:small-area} bounds the right side by
$C\xi^{3-\eps}$. At time $4$ the pulled-back volume density is
uniformly comparable to $4^{n/2}\dV_g$.
Pulling \eqref{eq:bad-volume-flux} back to $K$ proves
\eqref{eq:avoidance}.
\end{proof}

\subsection{Regular minimizing geodesics}

\begin{lemma}[Compactness at fixed cutoff]
\label{lem:good-subsequence}
If $\xi>0$ is fixed and $y\in K$ satisfies $y\notin B_{i,\xi}$
for infinitely many $i$, then $y\in G_x(D_\xi)$.
\end{lemma}

\begin{proof}
Pull the corresponding minimizing curves back and use $s$ to obtain
$\gamma_i:[0,1]\to\Omega_\xi$, from $x$ to $y$.
The physical energy bound and the bounded coordinate drift on the
fixed compact set give $\int_0^1|\dot\gamma_i|_g^2\dd s\le C_\xi$.
By a standard compactness argument, a subsequence converges
uniformly and weakly in $H^1$ to $\gamma:[0,1]\to D_\xi$, and
\begin{equation}
 E_g(\gamma)\le\liminf_i E_g(\gamma_i).
 \label{eq:energy-lsc}
\end{equation}
The smooth coefficient convergence
\eqref{eq:lagrangian-convergence} and the energy bound imply
\begin{equation}
 \Act_i(\beta_i)=B(x,y)+E_g(\gamma_i)+o(1).
 \label{eq:action-limit}
\end{equation}
Indeed, the error is bounded by a coefficient tending to zero times
$\int_0^1(1+|\dot\gamma_i|+|\dot\gamma_i|^2)\dd s$.

Let $\sigma$ be any regular piecewise smooth curve from $x$ to $y$.
Its image and its soliton lift are compact regular sets, so for
large $i$ they belong to the same convergence domains even if
they are outside $D_\xi$. The $\mathcal L$-lengths of the lifted curves tend to
$B(x,y)+E_g(\sigma)$. Global minimality on the closed approximating Ricci flows
implies
\[
 \limsup_i\Act_i(\beta_i)\le B(x,y)+E_g(\sigma).
\]
Equations~\eqref{eq:energy-lsc} and \eqref{eq:action-limit} give
$E_g(\gamma)\le E_g(\sigma)$. Taking the infimum over all regular
competitors yields $E_g(\gamma)=d(x,y)^2$.
Its image is compact in the smooth manifold $\Reg'_{-1}$; local first
variation makes it a smooth constant-speed geodesic.
\end{proof}

\begin{proof}[Proof of Theorem~\ref{thm:quantitative}]
The domain construction gives \eqref{eq:Dxi-containment} after
$\xi_0$ is decreased to include the endpoint neighborhoods.
The set $E_\xi=K\setminus G_x(D_\xi)$ is Borel by
Lemma~\ref{lem:measurability}. Lemma~\ref{lem:good-subsequence} gives the
pointwise inclusion
\[
 E_\xi\subset\liminf_{i\to\infty}B_{i,\xi}
       =\bigcup_{N=1}^\infty\bigcap_{i\ge N}B_{i,\xi}.
\]
Fatou's lemma and \eqref{eq:avoidance} therefore give
\[
 \Vol_g(E_\xi)\le\Vol_g(\liminf_iB_{i,\xi})
       \le\liminf_i\Vol_g(B_{i,\xi})\le C\xi^{3-\eps}.
\]
\end{proof}

\begin{proof}[Proof of Theorem~\ref{thm:weak}]
Fix $x\in\Reg'_{-1}$ and a compact exhaustion $K_m\Subset\Reg'_{-1}$.
Apply Theorem~\ref{thm:quantitative} on each $K_m$, for example
with $\eps=1/2$, and choose any sequence $\xi_{m,j}\searrow0$.
The Borel set
\[
 N_{x,m}=\bigcap_{j=1}^\infty
             \bigl(K_m\setminus G_x(D_{\xi_{m,j}})\bigr)
\]
has zero volume, since its measure is bounded by
$C_m\xi_{m,j}^{5/2}$ for every $j$. Every point outside
$N_x^0=\bigcup_mN_{x,m}$ has a regular minimizing geodesic from
$x$. Its length equals $d(x,y)$, so it and all its subsegments
are ambient slice minimizers.

The distance function $d(x,\cdot)$ is locally Lipschitz on the
smooth manifold $\Reg'_{-1}$. Exclude $x$ and its null set of
nondifferentiability points. At each remaining endpoint, terminal
first variation, applied to endpoint variations with both signs,
identifies the terminal unit velocity of any regular minimizing
geodesic with $\nabla d(x,\cdot)$.
Backward uniqueness for the smooth geodesic equation identifies
all such curves. This proves the uniqueness assertion after
another null-set enlargement.
Finally, the soliton measure has a strictly positive smooth density
on $\Reg'_{-1}$, so it has the same null sets as $\dV_g$.
\end{proof}

\section{From weak convexity to strong convexity}
\label{sec:weak-applications}

We prove Theorem~\ref{thm:convexity} on the time $-1$ slice.
Write $\mathcal R=\Reg'_{-1}$ and $\mathcal S=X\setminus\mathcal R$.
Weak convexity first gives local volume and functional inequalities.
We then use the RCD heat flow to compare small balls along regular
minimizing geodesics. The entropy density identifies regular points in a limit
geodesic, and smooth geodesic uniqueness gives the assertion for every
ambient minimizer.
In this section, $p$ denotes a minimum point of $f$ on $X$.
Equip $X$ with the normalized soliton measure
\[
 \dd\mathfrak m=(4\pi)^{-n/2}e^{-f-\muo}\dV_g,
\]
where the measure is extended by zero across $X\setminus\mathcal R$.
Thus $\mathfrak m(X)=1$.
By Theorem~\ref{thm:completion}, $X$ is a proper length space and
$f$ extends continuously to $X$. On $\mathcal R$ we have
\begin{equation}
 \Ric_f:=\Ric+\nabla^2f=\frac{1}{2}g,\qquad
 \Delta_f:=\Delta-\langle\nabla f,\nabla\cdot\rangle,
 \qquad \Delta_f f=\frac{n}{2}-f.
 \label{eq:weighted-geometry}
\end{equation}

By Fang--Li~\cite[Lemma~D.1]{FangLiStructure} and
Theorem~\ref{thm:completion}, $(X,d,\mathfrak m)$ is an
$\mathrm{RCD}(1/2,\infty)$ space. We use the heat semigroup of this
completed metric measure space.

\subsection{Sharp potential estimates}

The quadratic potential estimates for smooth shrinkers originate
in Cao--Zhou~\cite{CaoZhou}; the dimensional form used below is
\cite[Lemma~2.1]{HaslhoferMuller}. On a metric soliton, writing $r$
for the distance from a regular center, Chan--Ma--Zhang proved the
lower bound $r^2/(8+\epsilon)-C(\epsilon)$ in
\cite[Theorem~1.1]{ChanMaZhangSolitons}. Fang--Li improved the
denominator to $4+\epsilon$ in
\cite[Lemma~7.12]{FangLiStructure}. Weak convexity recovers the
smooth-form estimate with coefficient $1/4$ and a dimensional
additive error in distance.

\begin{proposition}[Sharp potential bounds]
\label{prop:sharp-potential}
The function $f$ attains its minimum at some $p\in X$, which need
not belong to $\mathcal R$, and $0\le f(p)\le n/2$.
For every such minimum point and every $x\in X$,
\begin{equation}
 \frac14\bigl(d(p,x)-5n\bigr)_+^2
 \le f(x)\le
 \frac14\bigl(d(p,x)+\sqrt{2n}\bigr)^2,
 \label{eq:sharp-potential}
\end{equation}
where $a_+=\max\{a,0\}$.
\end{proposition}

\begin{proof}
We follow \cite[Lemma~2.1]{HaslhoferMuller}, explaining the changes
needed on $X$. The soliton identities make $\sqrt f$
$1/2$-Lipschitz on $X$, and \eqref{eq:quadratic-potential} makes
$f$ proper. Integrating $\Delta_f f=n/2-f$ gives
$\int_Xf\dd\mathfrak m=n/2$: first remove the singular cutoffs
of Lemma~\ref{lem:cutoff-package}, then exhaust by outer cutoffs.
The local gradient bound and the coarse growth estimates make the
respective errors vanish. Thus a minimum point $p$ exists,
$0\le f(p)\le n/2$, and the Lipschitz bound gives the upper estimate.

For each regular $x$ and almost every $y$, Theorem~\ref{thm:weak}
supplies a regular minimizing segment. The index-form argument
of \cite[Lemma~2.1]{HaslhoferMuller} applies along this segment and,
when $d(x,y)>2$, gives
\[
 d(x,y)\le2\sqrt{f(x)}+2\sqrt{f(y)}+4n-\frac23.
\]
Continuity extends this estimate to all such endpoints.
Approximating $p$ by regular basepoints and using $f(p)\le n/2$
yields the lower estimate in \eqref{eq:sharp-potential}; for
distances at most $2$, it is immediate.
\end{proof}

Equation~\eqref{eq:sharp-potential} implies
\begin{equation}
 \sup_{B_d(p,L)\cap\mathcal R}\bigl(f+|\nabla f|^2+\mathrm{R}\bigr)
       \le C(n,L).
 \label{eq:local-potential-control}
\end{equation}

\subsection{Volume comparison and the segment inequality}

\begin{lemma}[Local volume and segment comparison]
\label{lem:local-volume-segment}
For every $L>0$ there is $C=C(n,L)$ with the following properties.
If $0<r_1<r_2$ and $B_d(q,r_2)\subset B_d(p,L)$, then
\begin{equation}
 \frac{\mathfrak m(B_d(q,r_2))}{\mathfrak m(B_d(q,r_1))}
       \le e^{C(r_2-r_1)}\left(\frac{r_2}{r_1}\right)^n.
 \label{eq:local-volume-comparison}
\end{equation}
If $A_1,A_2\subset B_d(q,r)\cap\mathcal R$ are measurable and
$B_d(q,r)\subset B_d(p,L)$, then every nonnegative measurable
function $a$ on $B_d(q,3r)\cap\mathcal R$ satisfies
\begin{equation}
 \begin{split}
 &\int_{A_1\times A_2}\int_0^{d(x,y)}a(\gamma_{xy}(s))\dd s
                          \dd\mathfrak m(x)\dd\mathfrak m(y)\\
 &\hspace{1cm}\le Cr\bigl(\mathfrak m(A_1)+\mathfrak m(A_2)\bigr)
                          \int_{B_d(q,3r)}a\dd\mathfrak m,
 \end{split}
 \label{eq:local-segment-inequality}
\end{equation}
where $\gamma_{xy}$ is the unique regular minimizing geodesic for almost
every pair $(x,y)$.
\end{lemma}

\begin{proof}
Theorem~\ref{thm:weak} permits polar integration along regular
minimizing geodesics over a full-measure set of endpoints. For a regular
center $q$, put $d_q=d(q,\cdot)$. Weighted comparison gives
\[
 \Delta_fd_q\le\frac{n-1}{d_q}+C,\qquad
 \frac{J_f(b,\theta)}{J_f(a,\theta)}
       \le\left(\frac ba\right)^{n-1}e^{C(b-a)}
 \quad(0<a<b)
\]
along regular minimizing rays, by \eqref{eq:local-potential-control}
and \cite[Theorems~1.1(a) and~1.2(a)]{WeiWylie}.
The usual polar and midpoint integrations now prove the stated
estimates; see \cite[Theorems~2.5--2.6]{HuangLiWang},
\cite[Propositions~2.3 and~2.6]{ChenWangA}, and
\cite[Propositions~A.2--A.3]{LiWangCylinders}.
Approximation allows arbitrary centers in the volume comparison.
It also gives zero measure to spheres and, for $0<\alpha<1/2$
on a fixed bounded region,
\begin{equation}
 \frac{\mathfrak m(B_d(q,(1+\alpha)r))}
      {\mathfrak m(B_d(q,r))}\le1+C\alpha.
 \label{eq:thin-annulus}
\end{equation}

Radial integration by parts gives weak Laplacian comparison on $X$: the terminal contribution is $-\phi J_f\le0$
for $\phi\ge0$, also on rays ending at singular points. Hence, for
nonnegative compactly supported Lipschitz $\phi$ away from a regular
center $q$,
\[
 -\int_X\langle\nabla d_q,\nabla\phi\rangle\dd\mathfrak m
       \le\int_X\phi\left(\frac{n-1}{d_q}+C\right)\dd\mathfrak m.
\]
\end{proof}

\subsection{Local Poincar\'e and Sobolev inequalities}
\label{subsec:local-functional}

For every bounded open set $\Omega\subset X$, restriction gives
an isometric isomorphism
\[
 W^{1,2}(\Omega,d,\mathfrak m)
       \cong W^{1,2}(\Omega\cap\mathcal R,g,\mathfrak m)
\]
between the metric and classical weighted Sobolev spaces, preserving
gradient energy. To see this, for $0<\eps<1$,
Lemma~\ref{lem:cutoff-package} gives locally
\begin{equation}
 \int|\nabla\eta_\rho|^2\dd\mathfrak m
       \le C\rho^{2-\eps}\longrightarrow0.
 \label{eq:zero-capacity}
\end{equation}
Truncate the function first, and apply these cutoffs locally on
relatively compact subsets of $\Omega$. Their gradient errors
vanish by \eqref{eq:zero-capacity}. Smooth approximation on
$\mathcal R$ and a locally finite partition of unity in $\Omega$
then give the identification, as in
\cite[Proposition~A.5]{LiWangCylinders}. For a ball $B$, write
\[
 u_{B,\mathfrak m}=\fint_Bu\dd\mathfrak m
       :=\frac{1}{\mathfrak m(B)}\int_Bu\dd\mathfrak m.
\]

Chan--Ma--Zhang~\cite[Corollary~1.9]{ChanMaZhangSolitons}
established a global Sobolev inequality for functions compactly
supported in the regular part, with a scalar-curvature term in
the energy. Here we obtain local inequalities on the same ball
for functions in $W^{1,2}_{\mathrm{loc}}(X,d,\mathfrak m)$.

\begin{proposition}[Local functional inequalities]
\label{prop:local-functional}
For every $L>0$ there is $C=C(n,L)$ such that, for every
$B=B_d(q,r)\subset B_d(p,L)$ and
$u\in W^{1,2}_{\mathrm{loc}}(X,d,\mathfrak m)$,
\begin{equation}
 \fint_B|u-u_{B,\mathfrak m}|^2\dd\mathfrak m
       \le Cr^2\fint_B|\nabla u|^2\dd\mathfrak m.
 \label{eq:local-poincare}
\end{equation}
If $n>2$, then also
\begin{equation}
 \left(\fint_B|u|^{\frac{2n}{n-2}}\dd\mathfrak m\right)^{\frac{n-2}{n}}
       \le C\fint_B\bigl(r^2|\nabla u|^2+|u|^2\bigr)\dd\mathfrak m.
 \label{eq:local-sobolev}
\end{equation}
Both statements hold with $\dV_g$ in place of $\dd\mathfrak m$
and with the corresponding averages, after changing $C(n,L)$.
\end{proposition}

\begin{proof}
We follow \cite[Propositions~2.7--2.8]{ChenWangA},
\cite[Propositions~2.7 and~2.9]{HuangLiWang}, and
\cite[Propositions~A.6--A.7]{LiWangCylinders}, using the weighted
volume comparison of \cite[Theorem~1.2(a)]{WeiWylie} in the form
of Lemma~\ref{lem:local-volume-segment}.
Apply the segment inequality to $|\nabla u|^2$ to obtain the weak
local $L^2$-Poincar\'e inequality with an enlarged ball on the
right. The local doubling supplied by
\eqref{eq:local-volume-comparison} and ball chaining give \eqref{eq:local-poincare} with the same
ball on both sides. Here the chains follow minimizing segments toward
the center of $B$, with each auxiliary ball chosen smaller than its
distance to $X\setminus B$. This is the usual John-domain proof
for a metric ball in a geodesic space. The relative volume bound in
\eqref{eq:local-volume-comparison}, with exponent $n$, gives the
Sobolev--Poincar\'e estimate for $u-u_{B,\mathfrak m}$ on the
same chains. Adding the mean gives \eqref{eq:local-sobolev}, as in the local
Sobolev theory of Saloff-Coste~\cite{SaloffCoste1992}.
The preceding Sobolev identification extends the estimates to the stated function class.
Finally, \eqref{eq:local-potential-control} bounds the oscillation
of $f$ on $B_d(p,L)$. Weighted and unweighted averages are therefore
comparable with constants depending only on $n,L$, giving the
unweighted statements.
\end{proof}

\begin{remark}
\label{rem:rcd-local}
By Fang--Li~\cite[Lemma~D.1]{FangLiStructure} and the identification
in Theorem~\ref{thm:completion}, $(X,d,\mathfrak m)$ is an
$\mathrm{RCD}(1/2,\infty)$ space. This gives the global spectral gap
\[
 \int_X|u-u_X|^2\dd\mathfrak m
       \le2\int_X|\nabla u|^2\dd\mathfrak m,\qquad
 u_X=\int_Xu\dd\mathfrak m;
\]
see \cite{GigliKettererKuwadaOhta}. The local estimates
\eqref{eq:local-poincare}--\eqref{eq:local-sobolev} also use the
finite-dimensional volume and segment comparisons obtained from
weak convexity and the soliton identities.
\end{remark}

\subsection{Entropy density and regular points}
\label{subsec:entropy-density}

By Bamler~\cite[Theorem~2.18 and Theorem~2.16(e)]{BamlerStructure},
every metric tangent space at a point of $X$ is a noncollapsed
Ricci-flat metric cone. The corresponding parabolic blow-up is a
static cone, with smooth convergence on the regular part.

For $z\in X$, let $\mathcal N_z(\tau)$ be the pointed Nash entropy
of the tangent Ricci flow based at the spacetime point $(z,-1)$,
with backward time scale $\tau>0$. This is distinct from the entropy
of the weighted heat semigroup on the fixed slice. Set
\[
 \mathcal N_z(0)=\lim_{\tau\searrow0}\mathcal N_z(\tau).
\]
For a metric tangent cone $(C,o)$ at $z$, write
\[
 \vartheta(C)=\frac{\mathscr{H}^n(B_C(o,1))}{\omega_n},
 \qquad \omega_n=\Vol_{\mathbb R^n}(B(0,1)).
\]
Here $\mathscr{H}^n$ has the normalization agreeing with Riemannian
volume on the regular part.

\begin{lemma}[Entropy and cone volume]
\label{lem:entropy-density}
For every $z\in X$ and every tangent cone $C$ at $z$,
\begin{equation}
 \mathcal N_z(0)=\log\vartheta(C).
 \label{eq:entropy-density}
\end{equation}
In particular, the right-hand side is independent of the choice
of tangent cone. There is $\epsilon_n>0$ such that
\begin{equation}
 z\in\mathcal S\quad\Longrightarrow\quad
 \vartheta(C)\le e^{-\epsilon_n}<1.
 \label{eq:singular-density-gap}
\end{equation}
A point of $X$ is regular if and only if it has a Euclidean tangent
cone.
\end{lemma}

\begin{proof}
The static tangent soliton has potential $d_C(o,\cdot)^2/(4\tau)$
up to its normalization constant. Its vertex conjugate heat kernel is
\[
 K_C(o,w;\tau)=\vartheta(C)^{-1}(4\pi\tau)^{-n/2}
               \exp\left(-\frac{d_C(o,w)^2}{4\tau}\right).
\]
Indeed, cone polar integration gives
\[
 \int_C e^{-d_C(o,w)^2/(4\tau)}\dd\mathscr{H}^n(w)
       =\vartheta(C)(4\pi\tau)^{n/2},
 \qquad
 \int_C\frac{d_C(o,w)^2}{4\tau}K_C\dd\mathscr{H}^n(w)=\frac n2.
\]
Thus its pointed Nash entropy is $\log\vartheta(C)$.
Apply this calculation to a sequence of parabolic dilations whose
spatial limit is $C$. Entropy convergence for these dilations gives
\eqref{eq:entropy-density}; see the entropy convergence and tangent
flow results in \cite[Theorems~2.10--2.11]{BamlerStructure}.

At a regular point the small-scale entropy is zero. Conversely,
if $\mathcal N_z(0)>-\epsilon_n$, then at a sufficiently small
positive scale the entropy satisfies the regularity threshold in
\cite[Proposition~7.7]{FangLiStructure}, after decreasing
$\epsilon_n$ if necessary. Hence $z$ is regular. Its contrapositive
and \eqref{eq:entropy-density} give
\eqref{eq:singular-density-gap}. A Euclidean tangent cone has
$\vartheta=1$, so the same argument proves the last assertion.
\end{proof}

\subsection{Heat regularization on the completion}
\label{subsec:heat-regularity}

Let $P_s$ be the weighted heat semigroup on $(X,d,\mathfrak m)$,
where $s$ denotes heat time on the fixed slice. The Sobolev
identification above identifies its Dirichlet form with the closure
of $\int_{\mathcal R}|\nabla v|^2\dd\mathfrak m$. In particular,
its generator on $\mathcal R$ is $\Delta_f$, and positive-time heat
solutions are smooth there. The semigroup preserves constants and
satisfies the Bakry--\'Emery estimate
\begin{equation}
 |\nabla P_sv|^2\le e^{-s}P_s(|\nabla v|^2)
 \qquad\mathfrak m\text{-almost everywhere}.
 \label{eq:heat-gradient}
\end{equation}
See \cite[Section~6]{AmbrosioGigliSavare}.

\begin{lemma}[Bochner inequality on the completion]
\label{lem:global-bochner}
Let $v$ be bounded and Lipschitz on $X$, and set $h=P_sv$ for
$s>0$. Then $|\nabla^2h|\in L^2_{\mathrm{loc}}(X,\mathfrak m)$
and $|\nabla h|^2\in W^{1,2}_{\mathrm{loc}}(X,d,\mathfrak m)$.
Moreover,
\begin{equation}
 (\partial_s-\Delta_f)|\nabla h|^2
       \le-2|\nabla^2h|^2-|\nabla h|^2
 \label{eq:weighted-Bochner}
\end{equation}
in the distributional sense on $X\times(0,\infty)$. The Hessian is computed on
$\mathcal R$ and extended by zero across $\mathcal S$.
\end{lemma}

\begin{proof}
Write $q=|\nabla h|^2$. Equation~\eqref{eq:heat-gradient} gives
$|\nabla h|\le L:=\|\nabla v\|_\infty$. Since
$\mathfrak m(X)=1$, we have $v\in L^2(X,\mathfrak m)$.
Spectral calculus for the nonnegative self-adjoint operator
$-\Delta_f$ gives
\[
 \Delta_fh\in W^{1,2}(X,d,\mathfrak m),\qquad
 h\in C^1\bigl((0,\infty);W^{1,2}(X,d,\mathfrak m)\bigr),
 \qquad \partial_sh=\Delta_fh,
\]
with the relevant norms bounded on compact positive-time intervals.
On $\mathcal R$, the smooth weighted Bochner identity is
\[
 \frac12\Delta_fq
       =|\nabla^2h|^2+\frac12q
                     +\langle\nabla h,\nabla\Delta_fh\rangle.
\]

Choose a compactly supported Lipschitz function $0\le\chi\le1$
on $X$ and put $\zeta_\rho=\chi\eta_\rho$, with the singular
cutoffs restricted to the reference slice. Its support is compact
in $\mathcal R$, so integration by parts and
$|\nabla q|\le2L|\nabla^2h|$ give
\[
 \begin{split}
 \int_{\mathcal R}\zeta_\rho^2|\nabla^2h|^2\dd\mathfrak m
 &\le 2L\int_{\mathcal R}\zeta_\rho|\nabla\zeta_\rho|
                    |\nabla^2h|\dd\mathfrak m
       +\int_{\mathcal R}\zeta_\rho^2
                    |\nabla h|\,|\nabla\Delta_fh|\dd\mathfrak m\\
 &\le\frac12\int_{\mathcal R}\zeta_\rho^2|\nabla^2h|^2
                    \dd\mathfrak m
       +2L^2\int_{\mathcal R}|\nabla\zeta_\rho|^2\dd\mathfrak m
       +\|\nabla h\|_2\|\nabla\Delta_fh\|_2.
 \end{split}
\]
Here the norms are in $L^2(X,\mathfrak m)$. By
\eqref{eq:zero-capacity}, the cutoff energy on the right is bounded
independently of $\rho$. Absorption and Fatou's lemma prove
$|\nabla^2h|\in L^2_{\mathrm{loc}}(X,\mathfrak m)$ by taking
$\chi=1$ on any prescribed compact set. These bounds are uniform
on compact positive-time intervals. The bound for $\nabla q$
and the Sobolev identification in
Subsection~\ref{subsec:local-functional} then give
$q\in W^{1,2}_{\mathrm{loc}}(X,d,\mathfrak m)$.

Now test the smooth identity with $\phi\eta_\rho$, where
$\phi$ is compactly supported and Lipschitz on $X$. The only
additional gradient term vanishes, since
\[
 \left|\int_{\mathcal R}
       \phi\langle\nabla q,\nabla\eta_\rho\rangle
                         \dd\mathfrak m\right|
 \le\|\phi\|_\infty
       \|\nabla q\|_{L^2(\operatorname{supp}\phi)}
       \|\nabla\eta_\rho\|_{L^2(\operatorname{supp}\phi)}
 \longrightarrow0.
\]
All other terms pass to the limit by the established integrability.
Thus
\[
 -\frac12\int_X\langle\nabla q,\nabla\phi\rangle\dd\mathfrak m
   =\int_X\phi\bigl(|\nabla^2h|^2+\frac12q
                  +\langle\nabla h,\nabla\Delta_fh\rangle\bigr)
                  \dd\mathfrak m.
\]
Finally, the $W^{1,2}$ differentiability of $h$ gives
$\partial_sq=2\langle\nabla h,\nabla\Delta_fh\rangle$ in $L^1$.
Integrating against compactly supported time tests proves
\eqref{eq:weighted-Bochner}, in fact with equality.
\end{proof}

\begin{lemma}[Local heat estimates and cutoffs]
\label{lem:local-heat-tools}
For every $L>0$, there are $s_0=s_0(n,L)>0$ and
$r_0=r_0(n,L)\in(0,\min\{1,\sqrt{s_0/2}\}]$ such that
the heat kernel has local Gaussian upper and lower bounds on
$B_d(p,L)$ for $0<s\le s_0$, and the corresponding local
parabolic mean-value estimates hold there. The constants are
controlled by the local volume and Poincar\'e constants on
$B_d(p,4L+4)$. For $q\in B_d(p,L)$ and $0<r\le r_0$, there
are cutoffs $\chi$ supported in $B_d(q,r)$, equal to one on
$B_d(q,r/2)$, Lipschitz on $X$ and smooth on $\mathcal R$, such that
\begin{equation}
 0\le\chi\le1,\qquad
 r|\nabla\chi|+r^2|\Delta_f\chi|\le C.
 \label{eq:heat-cutoff}
\end{equation}
The Laplacian bound holds for the global generator.
\end{lemma}

\begin{proof}
The completed Dirichlet form is strongly local and regular.
Local doubling and the Poincar\'e inequality give the heat kernel
and mean-value estimates by the argument of
\cite[Theorem~3.1]{HuangLiWang}. Lemma~\ref{lem:global-bochner} justifies the corresponding local
gradient estimates, as in \cite[Lemma~3.5]{HuangLiWang}.

Apply the heat-cutoff construction of
\cite[Theorem~3.6]{HuangLiWang} to a Lipschitz bump, using
positive-time heat averages on a time interval comparable to
$r^2$. These averages belong to the global generator domain,
with gradient bounded by $C/r$ and Laplacian bounded by $C/r^2$.
The Gaussian bounds permit a smooth scalar truncation that
vanishes on a separating annulus, so its zero extension belongs to
the global generator domain. A bounded finite covering then gives
the required support and half-ball value. The generator chain rule
proves \eqref{eq:heat-cutoff} on $X$. Positive-time regularity gives
smoothness on $\mathcal R$.
\end{proof}

We now prove the scale-uniform estimates used in the geodesic
comparison. Fix a bounded region, a lower and an upper bound for
geodesic length, and a positive distance from the endpoints as a
fraction of that length. Rescale the length to one. The rescaled
weighted Ricci tensor is still nonnegative, and all constants below
are uniform in these parameters. Fix $0<\delta<1/10$ and a unit-speed
regular minimizing geodesic $\gamma:[0,1]\to\mathcal R$. Put
\[
 q_-=\gamma(0),\qquad q_+=\gamma(1),\qquad
 d^-=d(q_-,\cdot),\qquad d^+=1-d(q_+,\cdot).
\]
Set $e=d^--d^+$. Choose a nonnegative outer cutoff $\psi$, supported in a fixed
bounded region away from $q_-$ and $q_+$ and equal to one on a
neighborhood of the region used below. The cutoff may be chosen
with uniform gradient and Laplacian bounds using
Lemma~\ref{lem:local-heat-tools}. In particular, it equals one when
both endpoint distances lie between $\delta/4$ and $2$, and the
support lies where both are at least $\delta/8$. Define
\begin{equation}
 h_s^\pm=P_s(\psi d^\pm),\qquad
 e_s=h_s^--h_s^+=P_s(\psi e)\ge0.
 \label{eq:heat-distances}
\end{equation}
All balls below are small enough to lie in the region where $\psi=1$,
with a fixed buffer from its boundary.

\begin{lemma}[Integrated heat approximation]
\label{lem:heat-approximation}
There are $C$ and $r_0>0$, uniform in the parameters just fixed,
such that, for $0<r<r_0$, the following hold.
\begin{enumerate}[label=\textup{(\roman*)}]
\item On the buffered region, for small $s>0$,
\begin{equation}
 |h_s^\pm-d^\pm|\le Cs+e,
 \qquad |\nabla h_s^\pm|^2\le1+Cs.
 \label{eq:heat-pointwise}
\end{equation}
For $t\in[\delta,1-\delta]$,
\begin{equation}
 \fint_{B_d(\gamma(t),r)}e\dd\mathfrak m\le Cr^2.
 \label{eq:mean-excess}
\end{equation}
\item There is a single $s_r\in[r^2/2,2r^2]$ for which
\begin{equation}
 \int_\delta^{1-\delta}
 \fint_{B_d(\gamma(t),5r)}
 \bigl(|\nabla^2h_{s_r}^-|^2+|\nabla^2h_{s_r}^+|^2\bigr)
             \dd\mathfrak m\dd t\le C.
 \label{eq:integrated-hessian}
\end{equation}
\item Let $\sigma_z$ be a unit-speed regular minimizing geodesic from $q_-$ to a point
$z$ in the region under consideration. For every subinterval $J$
with $\sigma_z(J)$ contained in the buffered region,
\begin{equation}
 \int_J|\nabla h_{s_r}^--\nabla d^-|^2(\sigma_z(t))\dd t
                    \le C\bigl(e(z)+r^2\bigr).
 \label{eq:radial-energy}
\end{equation}
In particular, on a subinterval of length $v$ the corresponding
$L^1$ error is at most $C\sqrt v(\sqrt{e(z)}+r)$.
For a unit-speed regular minimizing geodesic from $q_+$ the analogous estimates hold
with $\nabla h_{s_r}^+-\nabla d^+$.
\end{enumerate}
\end{lemma}

\begin{proof}
Lemma~\ref{lem:local-volume-segment},
Proposition~\ref{prop:local-functional}, and
Lemma~\ref{lem:local-heat-tools}, together with the global Bochner
inequality \eqref{eq:weighted-Bochner}, provide the inputs for the
heat-approximation argument of \cite[Section~2]{ColdingNaber} and
\cite[Theorem~3.12 and Lemmas~3.13--3.16]{HuangLiWang}.
The same argument gives (i)--(iii) on the completed weighted space.
The integrated energy estimate for (ii) is
\[
 \int_{r^2/2}^{2r^2}\int_\delta^{1-\delta}
 \fint_{B_d(\gamma(t),5r)}
 \bigl(|\nabla^2h_s^-|^2+|\nabla^2h_s^+|^2\bigr)
             \dd\mathfrak m\dd t\dd s\le Cr^2.
\]
Since the heat-time interval has length $3r^2/2$, averaging gives
one $s_r\in[r^2/2,2r^2]$ for which
\eqref{eq:integrated-hessian} holds, after adjusting $C$.
\end{proof}

\subsection{Uniform comparison of small balls}
\label{subsec:ball-comparison}

Following the Colding--Naber argument, we prove a uniform modulus
of continuity for small balls along regular minimizing geodesics.
All balls carry the distance restricted from $X$.

\begin{proposition}[Small balls along regular minimizing geodesics]
\label{prop:small-ball-comparison}
Fix $D,L,\ell_0>0$ and $0<\delta<1/2$. For every $\epsilon>0$
there are $a_0,r_0>0$, depending only on these parameters, $n$,
and $\epsilon$, such that the following holds. If
$\gamma:[0,\ell]\to\mathcal R$ is a unit-speed minimizing geodesic
contained in $B_d(p,D)$, with $\ell_0\le\ell\le L$, then
\begin{equation}
 d_{\mathrm{GH}}\bigl(\overline B_d(\gamma(s),r),
             \overline B_d(\gamma(t),r)\bigr)\le\epsilon r
 \label{eq:small-ball-comparison}
\end{equation}
whenever $s,t\in[\delta\ell,(1-\delta)\ell]$,
$|s-t|\le a_0\ell$, and $0<r\le r_0\ell$.
In particular, the constants do not depend on the smallest
regularity scale along $\gamma$.
\end{proposition}

\begin{proof}[Sketch of proof]
Normalize $\ell=1$ and work on a fixed interior interval, with a
slightly larger interval available for the estimates. We follow
\cite[Lemmas~4.1,~4.2,~4.4 and Theorem~4.5]{HuangLiWang}.
By Theorem~\ref{thm:weak}, the inward radial flows toward $q_-$
and $q_+$ are defined along regular minimizing segments for almost
every initial point. Their vector fields are $-\nabla d^-$ and
$+\nabla d^+$, respectively. Weighted polar comparison gives,
for either flow $\Phi_v$ on an annulus away from its endpoint,
\begin{equation}
 \mathfrak m(\Phi_v A)\ge e^{-Cv}\mathfrak m(A),\qquad
 (\Phi_v)_\#(\mathfrak m|_A)\le e^{Cv}\mathfrak m.
 \label{eq:radial-compression}
\end{equation}
For $t\in(0,1)$, concatenate any ambient minimizing geodesic from
$q_-$ to $\gamma(t)$ with a short continuation of $\gamma$ beyond
$t$; the resulting curve is still minimizing.
At the regular junction it has no corner, so backward uniqueness
for the smooth geodesic equation along $\gamma$ identifies the
original curve with $\gamma|_{[0,t]}$.
Properness then implies that all minimizing geodesics from $q_-$
to nearby endpoints remain in a regular neighborhood of this compact
segment. Since $\gamma$ continues minimizing beyond $t$, there is
no conjugate point at $\gamma(t)$; the local inverse of the
exponential map therefore shows that $d^-$ is smooth near
$\gamma(t)$. The reversed argument applies to $q_+$, and compactness
gives a common smooth neighborhood along any closed interior
interval.

Now let $\Phi_v$ denote the inward flow toward $q_-$.
The Riccati identity, endpoint Laplacian comparison, and the Jacobi
equation give the usual core estimate
$d(\Phi_vz,\gamma(b-v))\le e^{C\sqrt v}d(z,\gamma(b))$.
The core radius may depend on $\gamma$, but the constant does not;
see the argument of \cite[Lemma~4.1]{HuangLiWang}.
The corresponding estimate toward $q_+$ follows by reversing
$\gamma$.

Stop the flows when they leave doubled moving balls. Under the
temporary half-volume bounds for the surviving points,
Lemma~\ref{lem:heat-approximation}, the segment inequality, and
\eqref{eq:radial-compression} control the averaged distance
distortion. Chebyshev's inequality gives a good set $T$ and partner
sets $T(x)$, whose relative complements have measure at most
$C\eta$, with pairwise distortion at most
$C\eta^{-2}r\sqrt v$ before exit. To remove the stopping assumption,
choose $0<\alpha\ll1$, $0<\xi\ll\alpha$, and
$0<\eta\ll\xi^n$. Connect good points through the finite chain of
radii $r,\xi r,\ldots,\xi^Jr$ ending inside the smooth core.
The distortion estimate prevents exit from the moving
$(1+\alpha)r$ balls for a time depending only on the uniform
parameters. Applying the opposite endpoint flow and
\eqref{eq:thin-annulus} strictly improves the half-volume bounds,
so this time is independent of $J$ and of the core radius.
This is the stopping argument of
\cite[Lemmas~4.2 and~4.4]{HuangLiWang}.

Relative volume comparison makes $T$ $C\eta^{1/n}r$-dense in the
initial ball. Choosing common partners extends the distortion
estimate to all pairs in $T$. The opposite endpoint flow and
\eqref{eq:radial-compression} also show that the image misses at
most a $C(\alpha+\eta+v)$ fraction of the target ball's measure.
Consequently, it is $C(\alpha+\eta+v)^{1/n}r$-dense there.
Moving image points outside the target ball into it by at most
$\alpha r$ gives
\[
 \begin{split}
 &r^{-1}d_{\mathrm{GH}}\bigl(\overline B_d(\gamma(b),r),
                    \overline B_d(\gamma(b-v),r)\bigr)\\
 &\qquad\le C\bigl(\eta^{-2}\sqrt v+\eta^{1/n}
                    +(\alpha+\eta+v)^{1/n}+\alpha\bigr).
 \end{split}
\]
Choose $\alpha$ small, then $\xi$ and $\eta$ as above, and finally
$v$ sufficiently small. This proves
\eqref{eq:small-ball-comparison} with the stated uniformity.
\end{proof}

\subsection{Regular convexity}
\label{subsec:regularity-propagation}

\begin{lemma}[Limit geodesics with a regular endpoint]
\label{lem:regular-limit-geodesic}
For every $x\in\mathcal R$ and $y\in X$, there is a minimizing
geodesic $\sigma:[0,1]\to X$ from $x$ to $y$ with
$\sigma([0,1))\subset\mathcal R$.
\end{lemma}

\begin{proof}
The statement is immediate for $y=x$. Otherwise, choose good
endpoints $y_j\to y$ from Theorem~\ref{thm:weak} and regular
minimizing geodesics $\gamma_j$ from $x$ to $y_j$, parametrized at constant
speed on $[0,1]$. The good endpoints are dense because their
complement has zero measure and $\mathfrak m$ has full support.
Properness gives a uniformly convergent subsequence with a minimizing
limit $\sigma$. The lengths converge to $d(x,y)>0$, and the curves
lie in a common bounded region.

Let $J=\{t\in(0,1):\sigma(t)\in\mathcal R\}$. This set is open
and nonempty, since $x$ is regular. We show that it is relatively
closed in $(0,1)$. Suppose $t_i\in J$ and $t_i\to t\in(0,1)$.
For any fixed $\epsilon>0$, Proposition~\ref{prop:small-ball-comparison}
gives uniform $a_0,r_0>0$ on an interior interval containing these
times. For $i$ large, pass its estimate from $\gamma_j$ to $\sigma$:
\begin{equation}
 d_{\mathrm{GH}}\bigl(\overline B_d(\sigma(t_i),r),
             \overline B_d(\sigma(t),r)\bigr)\le\epsilon r
 \quad(0<r<r_0).
 \label{eq:limit-ball-comparison}
\end{equation}
This passage is at a fixed radius. In a proper length space,
closed balls of the same radius converge in Hausdorff distance
when their centers converge, by moving points along minimizing
segments if necessary.

For this fixed $i$, smoothness at $\sigma(t_i)$ implies
\[
 \lim_{r\searrow0}r^{-1}d_{\mathrm{GH}}
 \bigl(\overline B_d(\sigma(t_i),r),
                    \overline B_{\mathbb R^n}(0,r)\bigr)=0.
\]
The triangle inequality and \eqref{eq:limit-ball-comparison} give
an upper limit at most $\epsilon$ for the analogous expression at
$\sigma(t)$. Since $\epsilon$ is arbitrary, every tangent cone at
$\sigma(t)$ has a unit ball isometric, as a metric space, to the
Euclidean unit ball. Its $n$-dimensional volume is therefore
$\omega_n$. Lemma~\ref{lem:entropy-density} gives
$\mathcal N_{\sigma(t)}(0)=0$, hence $\sigma(t)\in\mathcal R$.
This proves closedness. Consequently $J=(0,1)$, as required.
\end{proof}

\begin{proof}[Proof of Theorem~\ref{thm:convexity}]
Let $\gamma:[0,l]\to X$ be a nonconstant unit-speed minimizing
geodesic that meets $\mathcal R$. Openness gives $a\in(0,l)$ and
$h>0$ such that $\gamma([a-h,a+h])\subset\mathcal R$.
By Lemma~\ref{lem:regular-limit-geodesic}, there is a unit-speed
minimizer $\eta:[a,l]\to X$ from $\gamma(a)$ to $\gamma(l)$,
regular on $[a,l)$. Concatenating $\gamma|_{[a-h,a]}$ with $\eta$
gives a curve of length
\[
 h+(l-a)=d\bigl(\gamma(a-h),\gamma(l)\bigr),
\]
so the concatenation is minimizing. It has no corner at the regular
point $\gamma(a)$, and hence $\dot\eta(a)=\dot\gamma(a)$.
The smooth geodesic equation gives agreement near $a$ and continues
it along all of $\eta([a,l))$. More explicitly, a first time at
which agreement ceased would still be a regular point on $\eta$;
openness of $\mathcal R$ and local uniqueness would extend the
agreement beyond that time. Thus $\gamma((a,l))\subset\mathcal R$.
Apply the same argument to the reversed geodesic to obtain
$\gamma((0,a))\subset\mathcal R$.

This proves the assertion for every ambient minimizer that meets
the regular part. If both endpoints are regular, the whole curve
is a smooth minimizing $g$-geodesic. The homotheties
$\boldsymbol{\varphi}_\tau$ on the completed slices transfer the
conclusion to every negative time.
\end{proof}

\begingroup
\small
\bibliographystyle{alpha}
\bibliography{TangentFlows}

\vskip12pt
\noindent
Yu Li, Institute of Geometry and Physics, University of Science
and Technology of China, No.~96 Jinzhai Road, Hefei, Anhui Province,
230026, China; Hefei National Laboratory, No.~5099 West Wangjiang
Road, Hefei, Anhui Province, 230088, China.\\
E-mail: \href{mailto:yuli21@ustc.edu.cn}{yuli21@ustc.edu.cn}.

\medskip
\noindent
Bing Wang, Institute of Geometry and Physics, and School of
Mathematical Sciences, University of Science and Technology of
China, Hefei 230026, China; Hefei National Laboratory,
Hefei 230088, China.\\
E-mail: \href{mailto:topspin@ustc.edu.cn}{topspin@ustc.edu.cn}.
\endgroup

\end{document}